\documentclass[12pt]{amsart}
\usepackage{amsmath}
\usepackage{amssymb}
\usepackage{enumerate}
\usepackage{latexsym}
\usepackage{layout}
\usepackage{nicefrac}
\usepackage{pdfsync}
\usepackage{setspace}
\usepackage{graphicx}
\usepackage{}
\theoremstyle{plain}

\newtheorem{theorem}{Theorem}
\newtheorem{corollary}{Corollary}

\newtheorem{lemma}{Lemma}
\newtheorem{proposition}{Proposition}

\newcommand{\bndry}{b}
\newcommand{\B}{\mathbf B}
\newcommand{ \Rop}{\mathcal R}
\newcommand{\Ropa}{\widetilde\Rop}
\newcommand{\C}{\mathbb C}
\newcommand{\CI}{\mathbf C}
\newcommand{\Cn}{\mathbb C^n}
\newcommand{\CT}{\mathcal C}
\renewcommand{\d}{{\tt{d}}}
\newcommand{\dee}{\partial}
\newcommand{\deebar}{\overline\partial}
\newcommand{\dz}{\delta (z)}

\newcommand{\D}{\mathcal D}
\newcommand{\inl}{(}
\newcommand{\inr}{)_{_{\!\mathbb R}}}
\newcommand{\inrc}{)_{_{\!\mathbb C}}}

\renewcommand{\l}{\lambda}
\renewcommand{\L}{\mathcal L_\rho (w)}
\newcommand{\La}{\Lambda}
\newcommand{\m}{m}
\newcommand{\Nop}{\mathcal N}
\renewcommand{\prec}{\langle\,}
\newcommand{\pz}{\pi (z)}
\newcommand{\Rn}{\mathbb R^N}
\newcommand{\Eop}{\mathcal E}
\newcommand{\Eopa}{\widetilde\Eop}
\renewcommand{\succ}{\,\rangle}
\newcommand{\Tcw}{T^{\C}_w}
\newcommand{\U}{U_z}
\renewcommand{\wp}{\hat w}
\newcommand{\xp}{\hat x}

\newcommand{\z}{z}

\newcommand{\zp}{\z}

\theoremstyle{definition}
\newtheorem{definition}{Definition}

\numberwithin{equation}{section}
\numberwithin{lemma}{section}
\numberwithin{theorem}{section}
\numberwithin{proposition}{section}
\begin{document}
\title[Cauchy Integral]{The Cauchy Integral in $\mathbb C^n$ for domains with minimal smoothness}
\author[Lanzani and Stein]{Loredana Lanzani$^*$
and Elias M. Stein$^{**}$}
\thanks{$^*$ Supported by a National Science Foundation IRD plan, and in part by award DMS-1001304}
\thanks{$^{**}$ Supported in part by the National Science Foundation, 
DMS-0901040}
\address{
Dept. of Mathematics,       
University of Arkansas 
Fayetteville, AR 72701}
\address{
Dept. of Mathematics\\Princeton University 
\\Princeton, NJ   08544-100 USA }

  \email{lanzani@uark.edu,\; stein@math.princeton.edu}
\begin{abstract} 
We prove $L^p(\bndry\D)$-regularity of the Cauchy-Leray integral for bounded domains $\D\subset\Cn$ whose boundary satisfies the minimal regularity condition of class $C^{1,1}$, together with a naturally occurring notion of convexity. 
\end{abstract}
\maketitle
\section{Introduction}\label{S:intro}
The purpose of this paper is to study the Cauchy integral in several complex variables and in particular to 
establish its $L^2$ (and $L^p$) boundedness in the setting of minimal smoothness assumptions on the boundary of the domain. We take as our model the well-known one-dimensional theory of Calder\'on \cite{C}, Coifman-McIntosh-Meyer \cite{CMM} and David \cite{D}, and its key theorem concerning the 
 boundedness in $L^p$ of the Cauchy integral for a Lipschitz domain.
  Our goal is to find an extension to $\mathbb C^n$ of these results, and in doing so we see that the context $n>1$ requires that we recast the problem to take into account its geometric setting, and also overcome inherent difficulties that do not arise in the case $n=1$. To describe our results we begin by sketching the background.
\subsection{Situation for $n=1$} The initial result was the classical theorem of M. Riesz for the Cauchy integral of the disc (i.e. the Hilbert transform on the circle) which gave the boundedness on $L^p$, for $1<p<\infty$. The
 standard proofs which developed for this then allowed an extension to a corresponding result where the disc is replaced by a domain $\D$ whose boundary is relatively smooth, i.e. of class $C^{1+\epsilon}$, for $\epsilon >0$.  However going beyond that to the limiting case of regularity, namely $C^1$ and other variants ``near'' $C^1$, required further ideas. Incidentally, the techniques introduced in this connection led to significant developments in harmonic analysis such as the ``$T(1)$ theorem'', and various aspects of multilinear analysis and analytic capacity see e.g., \cite{C-1} and \cite{C-2}, \cite{Me-C} and \cite{T}. The importance of those advances
  suggests the natural question: {\em what might be the corresponding results for the Cauchy integral in several variables?}
 \subsection{Problem for $n>1$}  When we turn to higher dimensions we see at once two basic differences which are present, that have no analogue in one dimension.
 \medskip
 
 \noindent $\bullet$\quad  The role of pseudo-convexity. That the pseudo-convexity of the underlying domain
   is a prerequisite 
 can be understood from a variety of points of view,  one of which is discussed below. For us the key consequence of this is that this condition, which essentially involves two degrees of differentiability of the boundary, implies that the correct limiting condition of smoothness should be ``near'' $C^2$, as opposed to near $C^1$ in one dimension.\\

\noindent$\bullet$\quad   The fact that given a domain $\D$ there is an infinitude of different ``Cauchy integrals'' that present themselves, as opposed to when $n=1$. This raises the further problem of finding (or constructing) the Cauchy integral appropriate for each domain
   that will be considered. The starting point for such constructions is the Cauchy-Fantappi\'e formalism, which grants the following reproducing formula for appropriate holomorphic functions $f$ in a suitable domain $\D$:
 \begin{equation}\label{E:repr-holom}
f(z)\ \,  =\   \frac{1}{(2\pi i)^n}\!\!\int\limits_{w\in\bndry\D}\!\!\!\! f(w)\, \eta\!\wedge\!(\deebar_w\eta)^{n-1}\ =:\ \ \CI (f) (z), \quad z\in\D\, 
\end{equation}
where $\eta = \eta (w, z) =\sum_j\eta_j(w, z) dw_j$ is a ``generating form'', i.e. it 
satisfies the condition
$\sum_j\eta_j(w, z) (w_j-z_j) = 1$, while $(\deebar_w\eta)^{n-1} =
\deebar_w\eta \wedge \cdots \wedge \deebar_w\eta$, with $n-1$ factors. Note that when $n=1$ there is only one possible such form, namely $\eta(w, z) = dw/(w-z)$, 
but if $n>1$ these can be quite arbitrary. The possible choices of $\eta$ are substantially constrained by a crucial condition that must be required of a Cauchy integral \eqref{E:repr-holom}, in analogy with the situation when $n=1$: that $\CI (f)$ be holomorphic in $\D$, for any ``arbitrary'' $f$ on $\bndry\D$. 
This condition, which requires that $\eta \wedge\!(\deebar_w\eta)^{n-1}(w, z)$ be holomorphic in $z$, has a number of consequences: it of course excludes the Bochner-Martinelli integral (for which one has $\eta_j = (\overline{w}_j-\overline{z}_j)/|w-z|^2$); and if one further assumes that $\eta(w,z)$ itself is holomorphic in $z$, then since
$\sum_j\eta_j(w, z) (w_j-z_j) = 1$ it follows that $\eta (w,z)$ must be singular at $z=w\in\bndry\D$, which incidentally implies that $\D$ is pseudo-convex. This makes the construction of such $\eta$ quite problematic. 
In fact with the exception of several very specific examples, even when the domain
is smooth, such $\eta$'s have been constructed only for two broad classes: when $\D$ is convex, 
or when $\D$ is strongly pseudo-convex.

\subsection{The situation when $\D$ is smooth} When $\D$ is smooth and convex a direct and natural construction goes back to Leray \cite{L}; see also Aizenberg \cite{A}, \cite{HL},   \cite{K} and references therein. The resulting Cauchy-Leray integral is given by 
\begin{equation}\label{E:CLgen}
\eta (w, z) = \frac{\dee\rho (w)}{\Delta (w, z)}
\end{equation}
where $\rho$ is a defining function of the domain $\D$, and $\Delta (w, z) = \sum_j(w_j-z_j)\dee\rho (w)/\dee w_j$. Here $|\Delta (w, z)|>0$ for $z\in\D$ and $w\in\bndry\D$ (by the convexity of $\D$), and $\eta(w, z)$ is clearly holomorphic in $z$, for $z\in\D$.

The case when $\D$ is instead strongly pseudo-convex (and still smooth) is less direct. Starting with the Levi polynomial at $w\in\bndry\D$, in which $\Delta (w, z)$ is augmented by appropriate quadratic terms in $(w-z)$, Henkin \cite{He} and
Ramirez \cite{Ra} have constructed generating forms in this case; see also 
 \cite{KS-1}, \cite{HL}, \cite{Ra} and references therein. However the study of the mapping properties of the resulting operators has hitherto been restricted to the situation when $\bndry\D$ is relatively smooth (implicitly of class $C^{2+\epsilon}$). We now turn to our main result, which deals with the limiting situation at $\epsilon =0$.
\subsection{When $\D$ is of class $C^{1,1}$: main result} We make two assumptions on our domain $\D$ bearing on its regularity and the nature of its convexity, which are suggested by the facts discussed above. First, we suppose that $\D$ is of class $C^{1,1}$, that is, the first derivatives of its defining function are Lipschitz. Second, we assume that $\D$ is strongly $\mathbb C$-linearly convex, which requirement has a simple geometric characterization in terms of the distance of a point $z\in\D$ from the (maximal) complex 
subspace of the tangent space at $w\in\bndry\D$. This is equivalent to the condition that $|\Delta (w, z)|\geq c|w-z|^2$. (One should note that ``$\mathbb C$-linear convexity'' is essentially intermediate between convexity and pseudo-convexity as discussed in Section \ref{S:C-lin-convex}). With this, our main result is as follows:\\

{\em Under these assumptions, the Cauchy-Leray integral, when properly defined, gives a bounded mapping on $L^p(\bndry \D)$ to itself for $1<p<\infty$}.
\subsection{Some elements of the proof} One needs to deal with a series of issues that are not present in the case $n=1$. The first is the following restriction problem: how to define second derivatives of a $C^{1,1}$ function on $\bndry\D$. In fact, in the definition of the Cauchy-Leray integral (\eqref{E:repr-holom} and 
\eqref{E:CLgen}), and in all aspects of its analysis, second derivatives of the defining function $\rho$ restricted to $\bndry\D$ appear. However $\rho$ is of class $C^{1,1}$, and such functions have second derivatives (in the sense of distributions) that are $L^\infty$ functions. These are definable only almost everywhere in the ambient space $\mathbb C^n$, while $\bndry\D$ has measure zero, so in general these derivatives are not definable on $\bndry\D$. One gets around this obstacle by observing that the second derivatives arising  in the Cauchy-Fantappi\`e formula \eqref{E:repr-holom} are in fact ``tangential'', and from this one can show that they can be given a unique 
point-wise meaning (a.e. on $\bndry\D$) via a second-order Taylor expansion, or equivalently, by a suitable approximation process. 

The second characteristic fact about the Cauchy-Leray integral for $n>1$ that one exploits is that its kernel,
 $\Delta (w, z)^{-n}$ is essentially a derivative of a kernel with a singularity like $\Delta (w, z)^{-n+1}$. This fact, together with Stokes' theorem, allows one to estimate the action of the Cauchy-Leray integral on test ``bump functions'' and give the ``cancellation conditions'' that are basic in treating the Cauchy-Leray integral as a singular integral. This procedure would not be valid when $n=1$ because then instead of $\Delta (w, z)^{-n+1}$ a logarithmic term would appear.
 
 A third noteworthy aspect is the crucial role of the measure on $\bndry\D$ given by 
 $$
 d\lambda = \frac{1}{(2\pi i)^n}\dee\rho \wedge (\deebar\dee\rho)^{n-1}
 $$
 referred to as {\em Leray-Levi measure}. It appears at many points below regarding the nature of the operator $\CI$, and
  in dealing with the ``adjoint'' of $\CI$. However, the $C^{1,1}$ assumption and the strong $\mathbb C$-linear convexity imply 
 that $d\lambda$ is equivalent with the induced Lebesgue measure $d\sigma$, thus in the end the results proved for $L^p(d\lambda)$ hold as well for $L^p(d\sigma)$.
 
 \subsection{Further remarks}\label{SS:further} We make two further comments. The first bears on the sharpness of our results for the Cauchy-Leray integral and shows that neither assumption made about the domain $\D$ can be essentially relaxed. In fact, the results in \cite{BL} give two different examples of bounded convex (Reinhardt) domains for which the Cauchy-Leray integral is well-defined, 
 but is not bounded on $L^2$:
 \begin{itemize}
 \item a strongly $\C$-linearly convex domain, whose boundary is of class $C^{2-\epsilon}$.
 \item a convex domain whose boundary is of class $C^\infty$.
 \end{itemize}
 It is also interesting to compare the Cauchy-Leray integral with the (orthogonal) Cauchy-Szeg\H o projection. In one dimension their connection is quite close: for the unit disc the two are identical, and for smooth domains their difference is ``small'' in that it is a smoothing operator \cite{KS-2}. However the situation changes markedly when $n>1$. Recall that when $\D$ is smooth and strongly pseudo-convex, the asymptotic formula for the Cauchy-Szeg\H o kernel found by 
 Fefferman \cite{F} allows one to prove the $L^p$-boundedness of the operator; and if $\D$ is convex and of ``finite type'' 
 the work of McNeal \cite{Mc} (see also \cite{McS}) gives appropriate estimates for the Cauchy-Szeg\H o kernel that
 lead to $L^p$-boundedness. While the Cauchy-Leray integral and Cauchy-Szeg\H o projection agree when $\D$ is the unit ball, even when $\D$ is smooth these two operators are quite far apart. In particular we can assert that their difference 
 is a smoothing operator only in very special circumstances that occur in the situation where the ``tangential'' part of the matrix$\{\dee^2\rho/\dee z_j\dee z_k\}$ 
  vanishes on the boundary. Such domains are studied in \cite{Ma}, \cite{DT}, \cite{Bo-1} and \cite{Bo-2}, and correspond to complex ellipsoids. 
 
 The forthcoming paper \cite{LS-3} deals with the $L^p(\bndry\D)$ boundedness of the Cauchy-Szeg\H o projection for strongly pseudo-convex domains of class $C^2$. The proof involves elements of the present paper and the corresponding result for the Bergman projection \cite{LS-1}. A survey of the background of all these results is in \cite{LS-2}.

 \subsection{Organization of the paper} Section \ref{S:C11} deals with the restriction to $\bndry\D$ of appropriate second derivatives of any $C^{1,1}$ function $F$ given in $\Cn$. The restrictions arise in several alternative but equivalent ways. What matters here is that they are realized as bounded functions (or forms) on $\bndry\D$. At this stage of the exposition
  certain more standard arguments are deferred to an appendix, where they are briefly summarized. In the next section we define the notion of $\C$-linear convexity and elaborate some of its properties, including that of the Leray-Levi measure. Then in Section \ref{S:CLI} we show that the Cauchy-Leray integral is well-defined when $\D$ is merely of class $C^{1,1}$, and in particular that it is independent of the choice of defining function
  for $\D$. Thus the Cauchy-Leray integral
  is intrinsically given by the domain in question.
    The proof of the main theorem is begun in Section \ref{S:main-thm}, where it is shown that up to an acceptable error, $\CI (f) =  \Eop (df)$, where the kernel of $\Eop$ (the {\em essential part} of $\CI$) has a weaker singularity than that of $\CI$. From this the cancellation conditions of $\CI$ are deduced. The proof is concluded in Section \ref{S:proof-MT} using, as it turns out, a simplified form of the machinery of the $T(1)$ theorem of Coifman for spaces of homogeneous type, as in \cite{DJS} and \cite{C-2}. Finally, a second appendix presents a quantitative version of the implicit function theorem for $C^{1,1}$ functions that is needed throughout this work.
   
\section{Properties of $C^{1,1}$ functions}\label{S:C11}
Our aim here is to study the possibility of restricting the second derivatives of $C^{1,1}$ functions to appropriate submanifolds. First some definitions.

\subsection{The tangential Hessian}\label{SS:Thess}
Suppose $M$ is a submanifold of $\mathbb{R}^N$ and for the moment we shall assume that $M$ has at each point $w_0 \in M$ a tangent space $T_{w_0}$. If $F$ is a $C^1$ function on $\mathbb{R}^N$ then we can define, for each $w_0$, the \textquotedblleft tangential gradient\textquotedblright \ $\nabla_T F  (w_0 )$ as the vector in $T_{w_0}$ determined by $\displaystyle \inl \nabla_T F (w_0 ) , H  ) = \inl \nabla F  (w_0 ), H  )$ for every $H \in T_{w_0}$. Here  $\inl ,  )$ is the standard real inner product on $\mathbb{R}^N$, and $\nabla F$ denotes the usual gradient  of $F$ in $\Rn$.
Similarly if $F$ is of class $C^2$ in $\mathbb{R}^N$, we denote by $\nabla^2 F (w_0 )$ the symmetric matrix of second derivatives of $F$ at $w_0$, the Hessian.
 For $w_0 \in M$ we define the \textquotedblleft tangential Hessian\textquotedblright \ $\displaystyle \nabla^2_T F (w_0 )$ as the symmetric linear transformation on $T_{w_0}$ that satisfies
\begin{equation*}
   \big\inl  \nabla^2_T F  (w_0 )\, U, \ V\big ) = \big\inl \nabla^2 F  (w_0 )\, U, V \big )\quad \mbox{for all} \ \ \ U, V \in T_{w_0}.
\end{equation*}
 
\subsection{The class $C^{1,1}$} A bounded function $F$ defined in $\mathbb{R}^N$ is \textquotedblleft Lipschitz\textquotedblright \ if $| F(w) - F(z) | \leq C | w - z |$ for some $C$ and all $w, z \in \mathbb{R}^N$. This is equivalent to assuming that $\nabla F$, taken in the sense of distributions, is a vector of $L^\infty (\mathbb{R}^N )$ functions. Similarly a bounded 
  $C^1$ function $F$ is of class $C^{1,1}$, if its first derivatives are Lipschitz functions, or alternatively if the Hessian matrix $\nabla^2 F$, taken in the sense of distributions, has entries that are $L^\infty$ functions. It is then convenient to define the $C^{1,1}$ norm of $F$ as $$\displaystyle \| F \|_{C^{1,1}} = \| F \| + \| \nabla F \| + \| \nabla^2 F \|$$
where $\|\cdot\|$ denotes the norm in $L^\infty (\mathbb R^N)$.

The submanifolds $M$ we shall be interested in are the boundaries of appropriate domains $\D \subset \mathbb{R}^N$. More precisely we consider a bounded domain $\D$ with defining function $\rho$, which means that $\D = \{ z \in \mathbb{R}^N : \ \rho (z) < 0 \}$ with $\rho: \mathbb R^N\to \mathbb R$. We shall then say that $\D$ is of class $C^{1,1}$, if $\rho$ is a function of class $C^{1,1}$, and $| \nabla \rho (w) | > 0$ \ \  whenever $w\in 
\{ w : \rho (w) = 0 \} =\bndry\D$.
It is clear that in this situation $M :=b\D$
 has a tangent space at each point.

\subsection{$\nabla^2_T F$ via Taylor expansion} Suppose we are given an $F \in C^{1,1} (\mathbb{R}^N )$, then its second derivatives (taken in the sense of distributions) are $L^\infty$ functions. The question we want to deal with is what meaning can be assigned to these quantities on $M =b\D$, taking into account that $L^\infty$ functions are defined only almost everywhere in $\mathbb{R}^N$, and $M$ has measure zero as a subset of $\mathbb{R}^N$. We will give various related answers to this question, but these have all in common that there is a well-defined meaning for those second derivatives that may be said to be \textquotedblleft tangential\textquotedblright . We deal first with an answer in terms of a second-order Taylor development.

Note that if $F \in C^2 (\mathbb{R}^N )$ and $w \in b\D$, then as a consequence of Taylor's formula we have
\begin{equation}\label{E:(1.1)}
   F(w + H ) - F(w ) = \big\inl \nabla_T F(w ), H \big ) + \frac{1}{2} \big\inl \nabla^2_T F  (w)H, H \big ) + o (| H |^2 )\ \ \mbox{for}\ \ H\in T_{w}
\end{equation} 
as $H \rightarrow 0$.

Now on $\bndry \D$ there is the measure $d \sigma$ induced by Lebesgue measure on $\mathbb{R}^N$. With this we then have the following.
\begin{proposition}\label{P:1}
   Suppose $F$ is of class $C^{1,1}$ on $\mathbb{R}^N$. Then, except for $w\in b\D$ lying in a set of $d\sigma$-measure zero, there is a symmetric linear transformation $\nabla^2_T F (w)$ on $T_{w}$ so that \eqref{E:(1.1)} holds.
\end{proposition}
Note that $\nabla^2_T F (w)$ is then uniquely determined by \eqref{E:(1.1)}.
To prove this proposition we need the second-order version of Rademacher's theorem.
\begin{lemma}\label{L:1} If $g$ is a $C^{1,1}$ function on $\mathbb{R}^{N-1}$, then  for $h\in\mathbb R^{N-1}$ 
      and for a.e. $x \in \mathbb{R}^{N-1}$
$$\displaystyle g(x + h) - g(x) = \inl \nabla g(x), h  ) + \frac{1}{2} \inl \nabla^2 g(x) h, h  ) + o (| h |^2 ),\quad \mbox{as}\ \ h \rightarrow 0.$$ 
The components of $\nabla g$ and $\nabla^2 g$ are the $L^\infty$ functions that arise when the respective derivatives are taken in the sense of distributions.
\end{lemma}
For the proof, see \cite{CZ-1} and \cite[VIII.6.1]{S-1}.
\smallskip

Now by an appropriate partition of unity and the implicit function theorem applied to the equation $\rho (w) = 0$ (see Appendix II) we can reduce matters to the following situation: $F$ is supported in a (small) ball in $\mathbb{R}^N$; the coordinates for $w$ in $\mathbb{R}^N$ can be taken to be $w = (x,y)$ with $x \in \mathbb{R}^{N-1}, \ y \in \mathbb{R}$.
In this ball, $\D = \{ (x,y): y > \varphi (x) \}$, where $\varphi$ is a $C^{1,1}$ function on $\mathbb{R}^{N-1}$. 
Then $d \sigma = \left( 1 + | \nabla \varphi (x) |^2 \right)^{\nicefrac{1}{2}} dx$, with $dx$ Lebesgue measure on $\mathbb{R}^{N-1}$, and hence sets of measure zero with respect to $d \sigma$ are sets of measure zero with respect to Lebesgue measure on $\mathbb{R}^{N-1}$.\\

We define $f_0$ and $f_1$ to be the functions on $\mathbb{R}^{N-1}$ that come about by restricting $F$ and 
$\dee_y F$
respectively,  to $\bndry \D = \{ (x,y) : y = \varphi (x)\}$. They are given by 
\begin{equation}\label{E:def-f-0}
\displaystyle f_0 (x) = F(x, \varphi (x))\qquad 
 \mbox{and}\qquad
 \displaystyle f_1 (x) =
 \frac{\dee F}{\dee y} 
 (x, \varphi (x)).
 \end{equation}
 Hence $f_0$ is a $C^{1,1}$ function while $f_1$ is a Lipschitz function.\\

   Now if $H \in T_{w},$ $w = (x, \varphi (x) ) \in b\D$, then
   $$ H = ( h, h_N ),\ \mbox{where} \ h \in \mathbb{R}^{N-1}\ \mbox{and}\  h_N = \inl \nabla\varphi (x) , h  ).$$
   (Note that in the above we are using the notation $(,)$ to denote two different things: a pair in the first occurrence, and
    the  inner product in $\mathbb R^{N-1}$ in the second occurrence.)
So 
$$F(w + H ) - F (w ) = I + II$$ where
$$
I = F(x + h , \ \varphi (x+h)) - F(x, \varphi (x))\qquad \mbox{and}$$ 
$$II = F(x + h , \ \varphi (x) + h_N ) \ - F(x + h, \varphi (x + h) ).$$

Observe that $I = f_0 (x + h) - f_0 (x)$, so applying 
Lemma \ref{L:1} to $g=f_0$
gives us
\begin{equation}\label{E:I}
   I = \inl\nabla f_0 (x), h  ) + \frac{1}{2} \inl \nabla^2 f_0 (x) h, h  ) + o (| h |^2 )\quad \mbox{for a.e.}\ \ x \in \mathbb{R}^{N-1}.
\end{equation}
However,
\begin{equation*}
   II = F (x + h, \ \varphi (x+h) + \delta ) - F(x + h, \ \varphi (x + h) )
\end{equation*}
where $$\delta = \varphi (x) + \inl \nabla \varphi (x), h  ) - \varphi (x + h).$$
Lemma \ref {L:1} applied to $g = \varphi$
 yields 
\begin{equation*}
   \delta = - \frac{1}{2} \inl \nabla \varphi (x) h,h  ) + o (| h |^2 ) , 
   \ \ \  \text{ for } a.e. \ \ \ x \in \mathbb{R}^{N-1}. 
\end{equation*}

The fact that $F$ is of class $C^{1,1}$ then shows that (uniformly in $h$),
\begin{equation*}
   II \ = \ \frac{\dee F}{\dee y} (x + h, \varphi (x + h)) \cdot \delta + O (\delta^2 )\, =\,
  f_1 (x+h)\delta + O(\delta^2 ) \, .
\end{equation*}
However by the Lipschitz character of $f_1$, we have  $f_1 (x+h) - f_1 (x) = O(|h|)$.

Altogether then 
\begin{equation}\label{E:(1.3)}
   II = -f_1 (x) \cdot \frac{1}{2} \inl \nabla^2 \varphi (x) h, h  ) + o (| h|^2 ) + O (| h | \delta ) + O (\delta^2 )
\end{equation} 
But the last two terms are $O(| h|^3 )$, so the addition of \eqref{E:I} and  \eqref{E:(1.3)} gives us the desired result \eqref{E:(1.1)} with
\begin{equation*}\label{E:(1.4)}
   \inl \nabla_T F\, (w), H  ) = \inl \nabla f_0  (x), h ) \quad \text{ and}
\end{equation*} 
\begin{equation}\label{E:(1.5)}
\inl \nabla_T^2 F\, (w) H, H   ) = 
 \inl \big( \nabla^2 f_0 (x) - f_1 (x) \nabla^2 \varphi (x)\big)h, h  )  \, ,
\end{equation}
where the first inner product is taken in $\mathbb R^N$ and the second is over $\mathbb R^{N-1}$.
More generally, if $ U$ and $V$ are in $T_{w}$ with $w=(x, \varphi (x))$ and
$U= (u, u_N)$, $V=(v, v_N)$, then for a.e. such $w$,
\begin{equation*}\label{E:(1.5)a}
 \inl \nabla_T^2 F\, (w) U, V   ) = 
 \inl \big( \nabla^2 f_0 (x) - f_1 (x) \nabla^2 \varphi (x)\big)u, v  )  \, .
\end{equation*} 
It is clear from this that the components of $\nabla_T^2 F$ are functions that belong to $L^\infty(\bndry\D, d\sigma)$.
\subsection{$\nabla^2_T F$ via approximations}\label{SS:1.4}
The above proposition gives us our first version of the existence of $\nabla^2_T F$ almost everywhere on $\bndry \D$. The next is in terms of approximations.

\begin{proposition}\label{P:2}
Suppose $F$ is a $C^{1,1}$ function on $\mathbb{R}^N$ and $\D$ is, as above, a domain of class $C^{1,1}$. Then there exists a sequence $\{ F_k \}$ of $C^\infty$ functions on $\mathbb{R}^N$ so that 
\begin{enumerate}
   \item The $\{ F_k \}$ have bounded $C^{1,1}$ norms, uniformly in $k$.
\item $F_k$ and $\nabla F_k$ converge uniformly on $\bndry \D$ to $F$ and $\nabla F$, respectively, as $k \rightarrow \infty$. More precisely, we have
\begin{equation}\label{E:unif-quantified}
F_k(w)-F(w)= O(1/k^2);\ \ \nabla F_k(w) - \nabla F(w) = O(1/k),\quad \mbox{if}\ w\in\bndry\D.
\end{equation}
\item $\nabla^2_T F_k$ converges to  $\nabla^2_T F$ almost everywhere on $\bndry \D$.
\end{enumerate}
\end{proposition}

To prove the proposition we need two approximation lemmas; these are stated below, but we defer the proofs of these rather technical facts to an appendix, so as not to interrupt the line of argument. We fix a $C^\infty$ function $\eta$ on $\mathbb{R}^{N-1}$, supported in the unit ball, normalized so that 
$$\displaystyle \int\limits_{\mathbb{R}^{N-1}}\!\!\!\!\eta (x)\, dx = 1$$
 and form the approximation to the identity $\{ \eta_\epsilon \}$, where 
 $$\displaystyle \eta_\epsilon (x) = \epsilon^{-N+1} \eta\left(\frac{x}{\epsilon}\right),\quad x\in\mathbb R^{N-1}.$$
 It will be important to take the precaution of requiring that $\eta$ is an \textit{even} function.
\begin{lemma}\label{L:2}
Assume $g \in C^{1,1} (\mathbb{R}^{N-1})$, and set 
$g_k = g \ast \eta_{\nicefrac{1}{k}}.$
   Then
\begin{enumerate}
   \item[(a)\ ] Each $g_k$ is a $C^\infty$ function.
\item[(b)\ ] $\| g - g_k \| = O (\nicefrac{1}{k^2} ), \ \  \| \nabla g - \nabla{g_k} \| = O(\nicefrac{1}{k} )$, and\ $\| \nabla^2 g_k \| = O(1), \text{ as } \ k \rightarrow \infty$.
\item[(c)\ ] $\nabla^2 g_k (x) \rightarrow \nabla^2 g (x)$, as $k\to\infty$,  for a.e.
 $x \in \mathbb{R}^{N-1}$.
\end{enumerate}
\end{lemma}

Here and immediately below, 
 $\| \cdot \|$ indicates the norm in $L^\infty (\mathbb{R}^{N-1})$.

\noindent The second lemma deals with a Lipschitz function $g$, and here we approximate by a \textquotedblleft double smoothing\textquotedblright . As before we write 
$g_k (x) = g \ast \eta_{\nicefrac{1}{k}}(x)$ and now define the function
\begin{equation*}
g_k (x, y) =\left\{
\begin{array}{lcl}
\left(g_k \ast \eta_{| y |}\right)  (x) = (g \ast \eta_{\nicefrac{1}{k}} \ast \eta_{|y|} )(x),&\mbox{when}&y\neq 0\\
&&\\
g_k(x) = \left(g \ast \eta_{\nicefrac{1}{k}}\right)(x),&\mbox{when}&y=0
\end{array}
\right.
\end{equation*}
  It will be convenient to denote by $D$ any first-order derivative in $x$ or $y$, and by $D^2$ any second-order derivative in these variables.

\begin{lemma}\label{L:3}
 Assume $g$ is a Lipschitz function on $\mathbb{R}^{N-1}$. Then
   \begin{enumerate}
      \item[(a)\ ]  Each $g_k (x,y)$ is a $C^\infty$ function on $\mathbb{R}^N$.
\item[(b)\ ] $\| g_k (\cdot , y) \| + \| ( Dg_{k} )(\cdot , y) \| = O(1) ,$
uniformly in $k$ and $y$.
\item[(c)\ ] $\| D^2 g_k (\cdot , y ) \| = O (\min \{k, \nicefrac{1}{| y |} \})$, uniformly in $k$ and $y$.
\item[(d)\ ] For almost every $x \in \mathbb{R}^{N-1}$, we have that 
$$
\dee_{x_j}g_k
(x,y) \to 
\dee_{x_j}g
 (x),\quad 
\dee_y\,g_k
 (x,y) 
\rightarrow 0,\ \ \mbox{as}\ k \rightarrow \infty, \ \ \mbox{and}\ | y | \leq \nicefrac{1}{k}.$$
   \end{enumerate}
\end{lemma}

\subsection{Proof of Proposition \ref{P:2}}\label{SS:1.6} 
As in Proposition \ref{P:1}, we can reduce matters to the situation where $F$ is supported in a ball in which a coordinate system $(x,y) \in \mathbb{R}^{N-1} \times \mathbb{R}$ is given so that in this ball, $\bndry \D$ is given by $\{ ( x,y): y = \varphi (x) \}$ with $\varphi \in C^{1,1} \left( \mathbb{R}^{N-1} \right)$. As before we define the restrictions $f_0$ and $f_1$ 
as in \eqref{E:def-f-0}.
 We also let 
\begin{equation}\label{E:aux-6}
f^k_0 (x)= (f_0 \ast \eta_{\nicefrac{1}{k}})(x) , \quad 
f^k_1(x) = (f_1 \ast \eta_{\nicefrac{1}{k}})(x), \quad\mbox{and}\quad 
\varphi_k (x) = (\varphi \ast \eta_{\nicefrac{1}{k}})(x).
\end{equation}
 In addition, we define the double smoothing $f_1^k (x,y)$ of $f_1$ by 
 \begin{equation*}\label{E:aux-7}
  f^k_1 (x, y) = 
  \left\{\begin{array}{rcl}
    f^k_1 \ast \eta_{| y |} 
  & \text{ if } & y \neq 0 \\
  \\
   f^k_1 (x) & \mbox{if}& y=0
   \end{array}\right.
  \end{equation*}

Then as our approximation to $F$ we take $\{ F_k \}$ given by
\begin{equation}\label{E:(1.9)}
   F_k (x,y) = f^k_0 (x) + \gamma (y - \varphi_k (x)) f^k_1 (x,y - \varphi_k (x)).
\end{equation} 
Here $\gamma$ is a fixed $C^\infty$ function on $\mathbb{R}$ of compact support, with $\gamma (t) = t$ when $| t | \leq 1$.\\

That the $F_k$ are $C^\infty$ on $\mathbb{R}^N$ follows immediately from {\em (a)} of Lemmas \ref{L:2} and \ref{L:3}, with $g = f_0 , \ \varphi , \text{ or } f_1$. Also it is not difficult to see that {\em (b)} of Lemma \ref{L:2} and {\em (b)} and {\em (c)} of Lemma \ref{L:3} give that $F_k , \nabla F_k$, and $\nabla^2 F_k$ are uniformly bounded in $\mathbb{R}^N$ and hence the $F_k$ have bounded $C^{1,1}$ norms.
Next we see that the $F_k$ (and their first derivatives) converge uniformly on $\bndry \D$ to $F$ (and its corresponding first derivatives). In fact 
$\dee_{x_j}\!F_k$
converges uniformly to 
$$
\displaystyle \frac{\dee f_0}{\dee x_j} - \frac{\dee \varphi}{\dee x_j}f_1, \ \ \text{ when }\ y = \varphi (x)
$$ 
by {\em (b)} of Lemma \ref{L:2} applied to $g = \varphi$, and {\em (b)} of Lemma \ref{L:3} applied to $g=f_1$. For the same reasons, 
$\dee_yF_k$
converges uniformly to 
$f_1=\dee_y F$,
when $y = \varphi (x).$

Finally we come to the proof that $\displaystyle \nabla^2_T F_k \rightarrow \nabla^2_T F$, almost everywhere on $b\D$. Let us set 
$$
\widetilde{f}_0^k (x) = F_k (x, \varphi (x))\quad\mbox{ and}\quad 
\widetilde{f}_1^k (x) =  \frac{\dee F_k}{\dee y}(x, \varphi (x)).
$$
Then according to \eqref{E:(1.5)} it suffices to see that
\begin{equation}\label{E:(1.10)}
      \nabla^2 \widetilde{f}^k_0 (x) \rightarrow \nabla^2 f_0 (x)   \quad \mbox{and}\quad
\widetilde{f}_1^k (x) \rightarrow f_1 (x) \quad \mbox{for a.e.}\ \ x\in\mathbb R^{N-1},\ \mbox{ as }\ \ k \rightarrow \infty .
\end{equation} 
But by \eqref{E:(1.9)}, 
\begin{equation*}
   \widetilde{f}^k_0 (x) = f^k_0 (x) + \gamma ( \varphi (x) - \varphi_k (x)) \cdot f^k_1 (x, \varphi (x) - \varphi_k (x)) .
\end{equation*}
So to verify the first limit in \eqref{E:(1.10)}, it suffices to see that any second derivative of \linebreak
$\gamma (\varphi (x) - \varphi_k (x)) f^k_1 (x, \varphi (x) - \varphi_k (x))$ tends to zero almost everywhere. If both derivatives fall on $f^k_1 (x, \varphi (x) - \varphi_k (x))$ then the result is $O (k)$, by {\em (c)} of Lemma \ref{L:3}, while $\gamma ( \varphi (x) - \varphi_k (x)) = O \left( \nicefrac{1}{k^2} \right)$, because of {\em (b)} of Lemma \ref{L:2}, so this term tends uniformly to zero. When one derivative falls on $\gamma (\varphi (x) - \varphi_k (x))$ and the other on $f^k_1 \left( x, \varphi (x) - \varphi_k (x) \right)$ we again get a contribution that tends uniformly to zero, since $\nabla (\varphi (x) - \varphi_k (x) ) = O (\nicefrac{1}{k} )$. Finally, when both derivatives fall on $\gamma (\varphi (x) - \varphi_k (x))$, the corresponding term tends to zero almost everywhere, by {\em (c)} of Lemma \ref{L:2}, and the first limit in \eqref{E:(1.10)} is established. 

For the second limit we claim that here the convergence is in fact uniform.
  We have that 
\begin{equation*}
 \widetilde{f}^k_1 (x) = \frac{\dee F_k}{\dee y}(x, \varphi (x) ) =                                                                                                                                                                                                                      f^k_1 (x, \varphi (x) - \varphi_k (x)) + \gamma (\varphi (x) - \varphi_k (x) ) 
 \frac{\dee f^k_1 }{\dee y} (x, \varphi (x) - \varphi_k (x))                                                                                                                                                                                                                          \end{equation*}
Recalling that $\gamma (t) = t$ for $|t|<1$, we see that
 the second term is $O (\nicefrac{1}{k^2} ) \cdot O(k) = O (\nicefrac{1}{k} )$ for $k$ sufficiently large, by
{\em (b)} of Lemma \ref{L:2}  and {\em (c)} of 
Lemma \ref{L:3}. Moreover $f^k_1 (x, \varphi (x) - \varphi_k (x))$ tends uniformly to $f_1$, in view of {\em (b)} of Lemma \ref{L:3}. 
So  \eqref{E:(1.10)}  is proved and Proposition \ref{P:2} is established.

\subsection{Uniqueness}\label{SS:1.7}
Proposition \ref{P:2} gives us a particular sequence $\{ F_k \}$ of smooth functions so that $\nabla^2_T F_k$ converges  to $\nabla^2_T F$ almost everywhere on $\bndry \D$. In fact the convergence to this limit must hold for any appropriate approximating sequence.
\begin{corollary}\label{C:1}
   Suppose $\{ G_k \}$ is a sequence of $C^\infty$ functions that satisfies conditions (1) and (2) of Proposition \ref{P:2} with $G_k$ in place of $F_k$. Assume also that $\nabla^2_T G_k$ converges almost everywhere on $\bndry \D$. Then $\lim_{k \rightarrow \infty}\limits  \nabla^2_T G_k  = \nabla^2_T F$ almost everywhere on $b\D$. 
\end{corollary}

As above, we focus on a coordinate patch. Recalling the definitions of  $f_0$ and $f_1$, the restrictions of $F$ and $\dee F/\dee y$ to $b\D$, see \eqref{E:def-f-0},

we denote by $g^k_0 = G_k (x, \varphi (x))$ and 
$\displaystyle g^k_1 (x) =\dee_y G_k(x, \varphi (x))$ 
the restrictions of $G_k$ and $\dee _yG_k$ to $\bndry\D$.
However $g^k_0 (x)$ converges uniformly to $f_0 (x)$, 
 since $\nabla G_k$ converges uniformly to $\nabla F$ on $\bndry \D$. Hence by applying \eqref{E:(1.5)} to $F = G_k$, we see that 
 $\displaystyle \dee^2_{x_ix_j}g^k_0 (x)$ 
 converges to a limit $L_{ij} (x)$ for a.e.  $x \in \mathbb{R}^{N-1}$. Moreover 
$\displaystyle \sup_{x, k}\limits \big|\dee^2_{x_ix_j}g^k_0(x) \big| \leq C$,
since the $\{ G_k \}$ are assumed to have bounded $C^{1,1}$ norms. Hence for any test function $\psi$
\begin{align*}
   \int\limits_{\mathbb{R}^{N-1}}\!\!\!\! L_{ij} (x) \psi (x)\, dx &= \operatorname*{\lim}_k\int
    \limits\limits_{\mathbb{R}^{N-1}}\!\!\!\!
     \dee^2_{x_ix_j}g^k_0(x)\,
    \psi (x)\, dx\\
&= \operatorname*{\lim}_k\limits \int\limits_{\mathbb{R}^{N-1}}\!\!\!\! g^k_0\,
 \dee^2_{x_ix_j}\psi\,
 dx = \int\limits_{\mathbb{R}^{N-1}}\!\! \!\!f_0\,
   \dee^2_{x_ix_j}\psi\,
  dx = \int\limits_{\mathbb{R}^{N-1}}\!\!\! 
   \dee^2_{x_ix_j}f_0\,
   \psi \, dx .
\end{align*}
where $\dee^2_{x_ix_j} f_0$  are $L^\infty (\mathbb R^{N-1})$ functions which are the corresponding derivatives
of $f_0\in C^{1,1}(\mathbb R^{N-1})$ taken in the sense of distributions.
This implies $\displaystyle L_{ij} (x) = \dee^2_{x_ix_j} f_0(x)$, 
for a.e. $x \in \mathbb{R}^{N-1}$
and Corollary \ref{C:1} is proved.\\

It is interesting to point out that the approximation $\{ F_k \}$ given by \eqref{E:(1.9)} has the additional property that \textit{all} second-order derivatives of the $F_k$ converge almost everywhere on $b\D$. However we will see below that for general $\{ G_k \}$ falling in the scope of Corollary \ref{C:1}, this may not be true.
\begin{corollary}\label{C:2}
   Suppose $\{ F_k \}$ is given by  \eqref{E:(1.9)}. Then $\operatorname*{\lim}_{k \rightarrow \infty}\limits D^2 F_k$ exists $a.e. \text{ on }b\D$ \ for any second order derivative $D^2$.
  \end{corollary}
All the second derivatives of the first term of \eqref{E:(1.9)} have already been treated, 
so we turn to the second derivatives of $\displaystyle \gamma (y - \varphi_k (x)) f^k_1 (x, y - \varphi_k (x))$. Now if a second derivative falls on $\gamma (y - \varphi_k (x))$ or if a second derivative falls on $ f^k_1 (x, y - \varphi_k (x))$ then these contributions were taken care of above. What remains is to see that $\dee_{x_j}\!\left( f^k_1 (x, \varphi (x) - \varphi_k (x))\right)$ 
and 
$\dee_y\!\left(f^k_1 ( x, \varphi (x) - \varphi_k (x)) \right)$
each converge to limits as $k \rightarrow \infty$ .
This follows immediately from {\em (d)} in Lemma \ref{L:3}, with $g=f_1$, once we observe that, by Lemma \ref{L:2},  $| \varphi (x) - \varphi_k (x) | = O \left( \nicefrac{1}{k^2} \right) \leq \nicefrac{1}{k}$, for large $k$.\\

A simple example of a sequence $\{ G_k \}$ that approximates $F$ in the sense of 
Corollary \ref{C:1} but such that 
$\displaystyle \operatorname*{\lim}_{k \rightarrow \infty}\limits\dee^2_{yy} G_k\neq \operatorname*{\lim}_{k \rightarrow \infty}\limits \dee_{yy}^2 F_k$ 
a.e. on $b\D$, is given by
\begin{equation*}
   G_k = F_k + \left( \gamma (y - \varphi_k (x)) \right)^2 \cdot g \ast \eta_{\nicefrac{1}{k}} (x)
\end{equation*}
where $g$ is any Lipschitz function that does not vanish.

\subsection{Differential forms in $\mathbb{C}^n$}\label{SS:1.8}
We apply the above for certain differential forms arising in complex analysis. Our setting is $\mathbb{C}^n$, which we identify with $\mathbb{R}^N$, where $N=2n$. We begin by observing that for  the sequence $\{F_k\}$ and $F$ as in Proposition \ref{P:2} 
we have that
      \begin{equation*}\label{E:unif-conv-pull-back-first}
      j^\ast (d F_k) \to j^\ast (d F);\ 
      j^\ast (\dee F_k) \to j^\ast (\dee F);\ 
         j^\ast (\deebar F_k) \to j^\ast (\deebar F)
      \  \mbox{uniformly on}\  \bndry\D;
      \end{equation*}
where $j^\ast ( d F)$; $j^\ast ( \dee F)$ and $j^\ast ( \deebar F)$ are the 1-forms on 
$\bndry \D$ that arise as the pull-back of $d F$, $ \dee F$ and $\deebar F$, respectively,  via the inclusion map 
$$j : \bndry \D \hookrightarrow \mathbb{C}^n.$$ 

Suppose now that $F$ is a function of class $C^2$ on $\mathbb{C}^n$. Then $\overline{\dee} \dee F$ is a 2-form (of type (1,1)) on $\mathbb{C}^n$ whose coefficients are continuous functions. With $\D$ a domain of class $C^{1,1}$, we write $j^\ast (\overline{\dee} \dee F)$ for the 2-form on 
$\bndry \D$ that arises as the pull-back of $\overline{\dee} \dee F$ via the inclusion map. 
We now seek to define $j^\ast (\overline{\dee} \dee F)$ when $F$ is merely of class $C^{1,1} \text{ on } \mathbb{C}^n$. 
We begin by noting that if 
$F$ were of class $C^2$, then we would have the identity 
\begin{equation}\label{E:(1.11)}
   \int_{b\D}\limits\!\! j^\ast (\overline{\dee} \dee F) \wedge \psi =  \int_{\bndry \D}\limits\!\! j^\ast (\dee F) \wedge d\psi
\end{equation} 
for any $(2n-3)$-form $\psi$ on $b\D$ that is of class $C^1$. Indeed, since $\overline{\dee} \dee F = d \dee F$ and $dj^*= j^*d$,
  then the left-hand side of \eqref{E:(1.11)} is $\int_{\bndry \D}\limits \!d (j^\ast \dee F) \wedge \psi$, and applying Stokes' theorem to $M=b\D$
 (with $bM =bb\D=\emptyset$)  proves that this equals the right-hand side of \eqref{E:(1.11)}. With this in mind, the basic facts about the existence of $j^\ast (\overline{\dee} \dee F)$ when $F$ is only $C^{1,1}$ are given by the following:
\begin{proposition}\label{P:3}
   Suppose $F \in C^{1,1} (\mathbb C^n)$ (and $n\geq 2$). Then 
  there exists a (unique) 2-form on $\bndry \D$, which we write as $j^\ast (\overline{\dee} \dee F)$, whose coefficients are in 
   $L^\infty (b\D)$ and that satisfies \eqref{E:(1.11)} for all $(2n-3)$-forms $\psi$ on $\bndry\D$ that are of class $C^1$.
  \end{proposition}
\begin{proof} 
Suppose $\{F_k\}$ is the approximating sequence given by \eqref{E:(1.9)}. Then
by Corollary \ref{C:2},  the sequence $\{j^\ast (\overline{\dee} \dee F_k)\}$ 
converges almost everywhere on $\bndry\D$ (that is, every coefficient converges almost everywhere on $\bndry\D$), and by Proposition \ref{P:2} it does so boundedly. Since
 \eqref{E:(1.11)} is verified for each $F_k$ and 
$j^\ast (\dee F_k) \rightarrow j^\ast (\dee F)$ uniformly on $\bndry\D$, in the limit we get \eqref{E:(1.11)} for $F$, once we have defined 
\begin{equation*}\label{E:def-pullback-second}
j^\ast (\overline{\dee} \dee F) = \operatorname*{\lim}_{k \rightarrow \infty}\limits j^\ast (\overline{\dee} \dee F_k ).
\end{equation*}

This verifies the existence of $j^\ast ( \overline{\dee} \dee F)$. Its uniqueness, given $F$, is evident from \eqref{E:(1.11)}, since the set of forms on $b\D$ whose coefficients are in $C^1 (b\D)$ is dense in the $L^1(b\D, d\sigma)$-norm in the
 space of forms on $b\D$ whose coefficients are in $L^1(b\D, d\sigma)$.
\end{proof}

\indent We shall also want to consider the $2m$-form denoted by $(j^\ast \overline{\dee} \dee F)^m$ and given by $j^\ast ( \overline{\dee} \dee F) \wedge j^\ast (\overline{\dee} \dee F) \cdots \wedge j^\ast (\overline{\dee} \dee F)$ ($m$ wedge products), with $1 \leq m \leq n-1$.\\

We observe that 
\begin{equation}\label{E:(1.12)}
 \int_{\bndry \D}\limits (j^\ast \overline{\dee} \dee F)^m \wedge \psi = \operatorname*{\lim}_{k \rightarrow \infty}\limits \int_{\bndry \D}\limits (j^\ast \overline{\dee} \dee F_k )^m \wedge \psi                  
\end{equation} 
where $F_k$ is the approximating sequence \eqref{E:(1.9)} and $\psi$ is any $(2n-2m-1)$ form 
with integrable coefficients. This is because $( j^\ast  \overline{\dee} \dee F_k )^m$ converges to $(j^\ast \overline{\dee} \dee F)^m$ pointwise almost everywhere and boundedly. The case when $m=n-1$, which is of particular interest, will be discussed in Section 
\ref{S:CLI}
below.

\subsection{The case of Lipschitz functions}\label{SS:LF} 
The results above have analogues where the $C^{1,1}$ function $F$ is replaced by a Lipschitz function $G$. The proofs of the corresponding results are then quite a bit simpler. We present a version of these results that will be applied later.

\begin{proposition}\label{P:approx-LF}
Suppose $G$ is a Lipschitz function on $\Rn$ and $\D$ is as above. Then there exists a 
sequence $\{G_k\}$ of $C^\infty$ functions on $\Rn$ so that
\begin{enumerate}
      \item[(1)\ ] $\{G_k\}$,\ $\{\nabla G_k\}$ and 
      $\displaystyle{\{\frac1k\nabla^2 G_k\}}$ are each uniformly bounded in $x\in \Rn$ and $k$.
      \item[(2)\ ] $\{G_k\}$ converges uniformly to $G$ on $\bndry \D$.
      \item[(3)\ ] $j^*(d G_k)(w)$ converges for $\sigma$-a.e. $w\in\bndry\D$ to a limit
       which we write as $j^*(d G)(w)$.
      \item[(4)\ ] The limit in item (3) above is the unique 1-form on $\bndry \D$ that arises from any sequence $\{G'_k\}$ satisfying conditions (1)-(3) above.
      \end{enumerate}
\end{proposition}
\begin{proof} As in the proof of Proposition \ref{P:1} and what follows, we may assume that $G$ is supported in a ball centered at the origin in $\Rn$, and that a coordinate system 
$(x, y)\in\mathbb R^{N-1}\times \mathbb R$ has been chosen so that in this ball
$\bndry\D =\{(x, y)\ :\ y=\varphi (x)\}$, with $\varphi$ a $C^{1,1}$ function.
Set $g(x) = G(x, \varphi (x))$, and let
\begin{equation*}
G_k(x, y) = (g \ast \eta_{1/k})(x) 
\end{equation*}
where $\eta_{1/k}$ are the approximation of the identity that were defined in Lemma \ref{L:2}. (Note that here $G_k(x, y)$ does not depend on $y$.) Now the proofs of {\em (1) - (3)} are a repetition of what has been done before.
For example, to estimate $\nabla^2 G_k$ we write
\begin{equation*}
\dee^2_{x_i x_j} G_k (x, y) = 
\dee_{x_j} g*\dee_{x_i} \eta_{1/k}(x)
\end{equation*}
where the $\dee_{x_j} g$'s are the $L^\infty$ functions that arise as the derivatives of $g$ taken in the sense of distributions. One then notes that
\begin{equation*}
\int\limits_{\mathbb R^{N-1}}\!\!\! \left|\dee_{u_i} \eta_{1/k} (u)\right| du = c\,k.
\end{equation*}
To prove {\em (3)} we first observe that 
$\dee_{x_j} G_k = 
(\dee_{x_j} g)* \eta_{1/k},
$
and hence $\dee_{x_j} G_k$ converges to a limits as $k\to\infty$ at each point in the Lebesgue set of $\dee_{x_j} g$. This shows that $\nabla G_k (w)$ converges for $\sigma$-a.e. $w\in\bndry \D$. Conclusion {\em (3)} is also proved because
 the action of $j^*(dG_k)(w)$ on any vector $v$ that is tangent to $\bndry \D$ at $w$
 is given by the inner product
 $\inl\nabla G_k (w), v)$.
 Finally, suppose $j^*(d G)$ and 
$j^*(d G')$ are respectively the limits of two sequences that satisfy {\em (1) - (3)} above. By Stokes' theorem have
\begin{equation*}
\int\limits_{\bndry\D}\!\!
j^*(d G_k)\wedge \psi = - 
\int\limits_{\bndry\D}\!\!
j^*(G_k) \, d\psi
\end{equation*}
for each $C^1$ form $\psi$ on $\bndry\D$ of degree $N-2$, with a similar identity for $G'_k$. Since 
$G_k\to G$ and $G'_k\to G$ uniformly on $\bndry\D$, we get as a result that
\begin{equation*}
\int\limits_{\bndry\D}\!\!
j^*(d G)\wedge \psi =
\int\limits_{\bndry\D}\!\!
j^*(d G')\wedge \psi \quad \mbox{for any}\ \psi\ \mbox{as above},
\end{equation*}
which gives $j^*(dG)=j^*(dG')$, concluding the proof of the proposition.
\end{proof}

\section{Geometry of the boundary of a strongly $\mathbb C$-linearly convex domain}
\label{S:C-lin-convex}
In this section we define the notion of strong $\C$-linear convexity for a given, bounded domain 
$\D$ of class $C^1$; we then outline its properties under the further assumption that $\D$ be of class $C^{1,1}$. 
\subsection{Strong $\C$-linear convexity}\label{SS:strong C-lin-convex}
Before giving the basic definitions we review some preliminaries. Our underlying space is
$\Cn$, which we identify, as before, with $\mathbb R^N$, $N=2n$.
Writing
$z\in \Cn$ as $z=(z_1,\ldots, z_n)$ with $z_j=x_j+iy_j$, the identification of $\Cn$ with 
$\mathbb R^{2n}$ is given by taking 
$
(x_1,\ldots, x_n, y_1,\ldots, y_n)=(x, y)
$
to be the coordinates in $\mathbb R^{2n}$ of the point corresponding to $z$. 
In order to distinguish between the Euclidean inner product in $\mathbb R^{2n}$ and the Hermitian inner product in $\mathbb C^n$, here and in the sequel we will adopt the notations $(\cdot,\cdot\inr$ and $(\cdot,\cdot\inrc$, respectively.
(Note that the notation we introduce here, which will be used throughout the rest of this paper, differs from that of Section \ref{S:C11}, in that there $x$ and $y$ where elements in $\mathbb R^{N-1}$ and $\mathbb R$, respectively, while here and below, $x$ and $y$ are elements of 
$\mathbb R^n$.)

Next, if we consider
the scalar multiplication by $i$ in $\Cn$, then by the above identification
$iz=J(x, y)$,
where $J$ is the linear transformation on $\mathbb R^{2n}$ given by
$J(x, y) = (-y, x)$.
We also have, as in Section  \ref{S:C11}, domains $\D$ in $\Cn$ with defining function $\rho$.
 For each $w\in\bndry\D$ we have the tangent space $T_w$, and the notion we need here is that of the ``complex tangent space'',
which we denote by $\Tcw$: it is the complex sub-space of $T_w$ of (complex) dimension $n-1$ that is given by
\begin{equation*}\label{D:Tcw}
\Tcw =\{v:\ v\in T_w\quad \mbox{and}\quad iv\in T_w\}
\end{equation*}
where the action of $i$ is interpreted in the sense described above. One can easily verify that
\begin{equation}\label{E:Tcw-rep}
\Tcw=\{v:\ \langle\dee\rho (w), v\rangle =0\}
\end{equation}
where
\begin{equation*}
\dee\rho (w) =\left(\frac{\dee\rho}{\dee w_1}(w),\ldots, \frac{\dee\rho}{\dee w_n}(w)\right)
\end{equation*}
with $2\,\dee/\dee w_j =\dee/\dee x_j -i \dee/\dee y_j$ and $w_j=x_j+iy_j$,
and we have used the notation
\begin{equation}\label{D:non-herm-pairing}
\langle z, \zeta\rangle = \sum\limits_{j=1}^nz_j\,\zeta_j
\end{equation}
if $z=(z_1,\ldots, z_n)$ and $\zeta =(\zeta_1,\ldots, \zeta_n)$. In fact
 we have that
\begin{equation}\label{E:pairing-real-im}
\left\{ \begin{array}{lcr}
\inl\nabla\rho(w), v\inr &=& 2\,\mbox{Re}\langle\dee\rho (w), v\rangle\\
\\
\inl\nabla\rho(w), Jv\inr &=& -2\,\mbox{Im}\langle\dee\rho (w), v\rangle
\end{array}\right.
\end{equation}
However $\inl\nabla\rho (w), v\inr=0$ is the same as the assertion that $v$ is in the tangent space
$T_w$. Thus \eqref{E:Tcw-rep} is equivalent with the statement that both $v$ and $iv$ are in $T_w$ and this means that $v\in \Tcw$.

The following variant of the function that appeared above
\begin{equation*}\label{D:Delta}
\Delta (w, z) = \langle\dee\rho (w), w-z\rangle
\end{equation*}
will play a basic role in what follows, and in particular as a ``denominator'' in the Cauchy-Leray integral. Further properties of $\Delta (w, z)$ are reflected in the equivalence of the two conditions
 below for any bounded domain of class $C^1$.
  \begin{align}
 &\mbox{For some}\ c>0, &|\Delta(w, z)|&\geq c|w-z|^2,&\mbox{if}\ z\in\D\ \mbox{and}\ w\in\bndry \D\, .\label{E:Delta-bound}\\
 \notag\\
 &\mbox{For some}\ c'>0, & d^E(z, w+\Tcw)&\geq c'|w-z|^2,&\mbox{if}\ z\in\D\ \mbox{and}\ w\in\bndry \D\, .\label{E:dist-bound}
 \end{align}
 Here $d^E(z, w+\Tcw)$ denotes the Euclidean distance from $z$ to the affine subspace
 $w+\Tcw$. Note that while $\Tcw$ is the complex tangent space referred to the origin, $w+\Tcw$
 is its geometric realization as an affine space tangent to $\bndry D$ at $w$.
 
 To prove the equivalence of \eqref{E:Delta-bound} and \eqref{E:dist-bound} it suffices to establish the identity
 \begin{equation}\label{E:repr-dist}
 d^E(z, w+\Tcw) = 2\,\frac{|\Delta (w, z)|}{|\nabla\rho (w)|}
 \end{equation}
 Note that the boundedness of $\D$ implies uniform bounds from above and below for $|\nabla\rho|$.
 In order to see \eqref{E:repr-dist} we introduce a coordinate system (centered at $w\in\bndry \D$) which will be useful on several different occasions. The new coordinates, $(z', z_n)$ with
 $z'=(z_1,\ldots,z_{n-1})\in \mathbb C^{n-1}$ arise from the usual coordinates by a translation
  composed with a unitary map of $\Cn$ so that:
 {\tt{(i)}}, the coordinates $(0', 0)$ correspond to the point $w$;
  {\tt{(ii)}}, the derivative along the inner pointing normal at $w$ is
   $\dee/\dee y_n$;     
 {\tt{(iii)}}, $\dee\rho (w)/\dee z_j =0,$ for $j=1,\ldots, n-1$ and
 $\dee\rho (w)/\dee x_n=0$.
 As a result we get the following, {\tt{(iv)}}:
 $$\frac{\dee\rho}{\dee z_n}(w) =
  -\frac{i}{2}\,\frac{\dee\rho}{\dee y_n}(w), \quad 
    |\nabla\rho (w)| = -2\,\frac{\dee\rho}{\dee y_n}(w)\quad \mbox{and}\quad 
  T^{\C}_w = \{(z', 0)\ |\ z'\in\mathbb C^{n-1}\}.$$
  We remark that the coordinate system we construct is not uniquely determined by $w$. However, the coordinate $z_n$ is, and so is $z'=(z_1,\ldots, z_{n-1})$ up to a unitary equivalence in $\mathbb C^{n-1}$.\\
  
  To set up these new coordinates, we first identify $\nu_w$, the inward-pointing unit normal at $w\in\bndry \D$. Since $\rho (w)=0$ and $\rho<0$ inside $\D$, then $-\nabla\rho (w)$ (as a vector in $\mathbb R^{2n}$) is an inward-pointing normal vector at $w$, and hence
  \begin{equation*}\label{E:normal}
  \nu_w=-\frac{\nabla\rho (w)}{|\nabla\rho (w)|}
  \end{equation*}
 With our  identification of $\mathbb R^{2n}$ with $\Cn$ we also have
  $\nabla\rho (w) = 2\,\deebar\rho (w)$
   where 
 \begin{equation}\label{E:dee-deebar}
  \deebar\rho (w) = \left(\frac{\dee\rho}{\dee \overline w_1}(w),\ldots, \frac{\dee\rho}{\dee\overline w_n}(w)\right) = \overline{\dee\rho (w)}
\end{equation}
with $2\,\dee/\dee\bar{w}_j = \dee/\dee x_j + i \dee/\dee y_j$ and $w_j=x_j+iy_j$.
Thus, as a vector in $\Cn$ we have that
   \begin{equation*}
  \nu_w=-2\,\frac{\deebar\rho (w)}{|\nabla\rho (w)|}\, .
  \end{equation*}
Next we note that the standard Hermitian inner product 
on $\Cn$ is
\begin{equation}\label{E:herm-non-herm}
\inl z, \zeta \inrc = \langle z, \overline\zeta\rangle = \langle \overline\zeta, z\rangle
\end{equation}
 with $\langle\cdot, \cdot\rangle$ as in in \eqref{D:non-herm-pairing}. So by  
\eqref{E:Tcw-rep}, $\Tcw$ is the complex subspace of $\Cn$ that is orthogonal
(with respect to $\inl \cdot, \cdot\inrc$) to $\nu_w$. 
Pick any orthonormal basis $\{e_1,\ldots,e_{n-1}\}$ of $\Tcw$ and set 
\begin{equation}\label{D:last-coord}
\displaystyle{e_n= i\,\nu_w = -2i\,\frac{\deebar\rho (w)}{|\nabla\rho(w)|}}, \quad \mbox{and thus}\quad
  \deebar\rho (w) =\frac{i}{2}|\nabla\rho (w)|\,e_n\, .
  \end{equation}
  So $\{e_1,\ldots,e_n\}$ is an orthonormal basis of $\Cn$, and the new coordinate
  system centered at $w$ is defined by the equation
  \begin{equation}\label{E:z-w-repr}
  z-w = \sum\limits_{j=1}^nz_j\,e_j.
  \end{equation}
   Now, taking the above
    into account then
  \begin{equation*}
  \frac{\dee (z-w)}{\dee y_n} = i\, e_n = -\nu_w\quad \mbox{for any}\ z,\ \mbox{so in particular}\quad
  \frac{\dee (z-w)}{\dee y_n}\bigg |_{w} = -\nu_w\, .
  \end{equation*}
  Moreover by \eqref{E:dee-deebar} --
    \eqref{D:last-coord},
   \begin{equation*}
    \Delta (w, z) = -\,\langle \dee\rho (w), z-w\rangle = -\,\inl z-w,\, \deebar\rho (w)\inrc =
  \frac{i}{2}|\nabla\rho (w)|\inl z-w, e_n\inrc.
  \end{equation*}
  But by \eqref{E:z-w-repr}, we have $\inl z-w, e_n\inrc = z_n$. As a result
  \begin{equation}\label{E:Delta-repr}
  \Delta (w, z) = \frac{i}{2}|\nabla\rho (w)| z_n\, 
  \end{equation}
 which gives
  \begin{equation*}\label{E:Delta-repr-a}
  |z_n| = 2\,\frac{|\Delta (w, z)|}{|\nabla\rho (w)|}.
  \end{equation*}
  Observe also by \eqref{E:z-w-repr} and the orthogonality of $z_ne_n$ with $\sum\limits_{j=1}^{n-1}z_j\,e_j$, that the distance of $z-w$ to $\Tcw$ is exactly $|z_n|$. Combining this observation with \eqref{E:Delta-repr} gives us \eqref{E:repr-dist} and thus the equivalence of
  \eqref{E:Delta-bound} and \eqref{E:dist-bound}.
  \begin{definition} Let $\D$ be a bounded
  domain of class $C^1$.
  We say that $\D$ is {\em strongly $\C$-linearly convex}  if at any boundary point it  satisfies either of the two equivalent conditions \eqref{E:Delta-bound} or \eqref{E:dist-bound}.
  \end{definition}
    This version of convexity is related to certain separation properties of the domain from its complement by (real or complex) hyperplanes, for which see \cite{APS} and \cite[IV.4.6]{Ho}. Such connection is a consequence of identity \eqref{E:repr-dist}.

Furthermore, it can be seen that strong $\mathbb C$-linear convexity is implied by strong $R$-convexity in the sense of
Polovinkin \cite{P}.

  \subsection{More about strong $\C$-linear convexity}\label{SS:more-C-linear} Here we return to the assumption that the bounded domain $\D = \{\rho<0\}$ is of class $C^{1,1}$.
   Then in accordance with Proposition \ref{P:1} the tangential Hessian
  $\nabla^2_T\rho (w)$ is defined via \eqref{E:(1.1)} for $\sigma$-almost every $w\in\bndry \D$. If for each such $w$ we denote by $\nabla^2_{T^{\C}}\rho (w)$ the restriction of the quadratic form
  $\nabla^2_T\rho (w)$ on $T_w$ to $\Tcw$, then $\nabla^2_{T^{\C}}\rho (w)$ is again defined for
  $d\sigma$-almost every $w\in\bndry \D$. We call $\nabla^2_{T^{\C}}\rho (w)$ the {\em complex tangential Hessian of} $\rho$ {\em at} $w\in\bndry \D$.
  \begin{proposition}\label{P:3.1} Suppose that $\D$ is strongly $\C$-linearly convex and of class $C^{1,1}$. Then 
   $\nabla^2_{T^{\C}}\rho (w)$ is positive definite uniformly in $w\in\bndry\D$  except for a set with $d\sigma$-measure zero.
   That is,
  there is a constant $c>0$, so that
  \begin{equation}\label{E:cplx-tang-hess-positive-def}
    \inl \nabla^2_{T^{\C}}\rho (w) h, h\inr\geq c |h|^2\quad \mbox{for any}\ h\in \Tcw\quad \mbox{and for} \ d\sigma\mbox{-almost every}\ w\in\bndry \D.
  \end{equation}
  \end{proposition}
  
  \begin{proof} The proposition can be proved in three steps.
  
  \noindent {\em Step 1.}\quad We first note that by continuity the inequality \eqref{E:dist-bound}
  extends to all $z\in\overline{\D}$ (with same constant $c'$). It follows that if $z\in\overline{\D}$ and $w\in\bndry\D$, but $z\neq w$, then $z-w\notin \Tcw$. So if $w\in\bndry\D$ and $h\in\Tcw$ with $h\neq 0$, then $z:=w+h\notin \overline{\D}$ and so 
    \begin{equation}\label{E:notinD}
  \rho(w+h)>0\, .
  \end{equation}
  \noindent {\em Step 2.}\quad Next, we consider $\nu_w$, the inward-pointing normal unit vector at $w$. With $w$ fixed and $|h|$ sufficiently small, we claim there is a smallest strictly positive
   $\alpha$, so that 
    $w+h+\alpha\nu_w\in\bndry\D$, that is $\rho(w+h+\alpha\nu_w)=0$, see 
    Figure 1. below.
  \begin{figure}[h!]
 \centering 
 \includegraphics
 {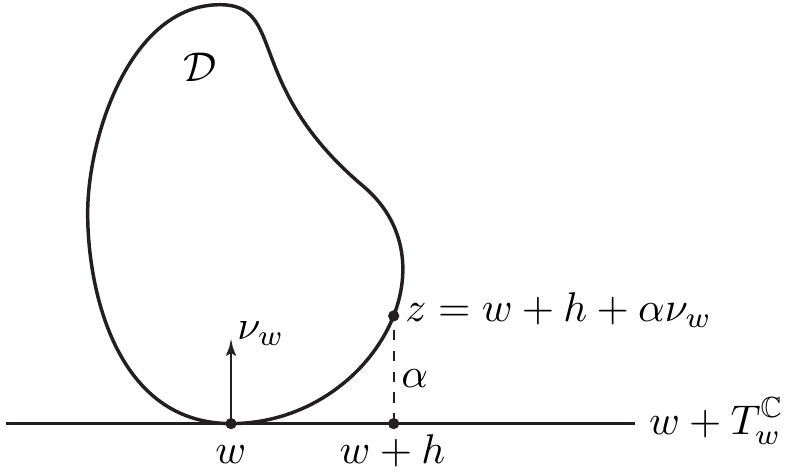}
\caption{}
\label{fig:Figure1}
\end{figure}
  
  Indeed, since $|h|$ is small and $w$ is in 
   $\bndry\D$,  then $w+h+\beta\nu_w$ must be inside of $\D$ if $\beta$ is sufficiently large and positive, hence 
  $\rho(w+h+\beta\nu_w)<0$. Continuity and \eqref{E:notinD} then guarantee the existence of such an $\alpha$. Now if we set $z:= w+h+\alpha\nu_w\in\bndry\D$,
  by the condition $h\in\Tcw$ and the orthogonality of $\nu_w$ to $\Tcw$, we have that
  \begin{equation*}
  \alpha = d^E(z, w+\Tcw)
  \end{equation*}
  and also $|z-w|^2 = \alpha^2 + |h|^2$. Therefore \eqref{E:dist-bound} implies
  \begin{equation}\label{E:alpha-bound}
  \alpha> c'|h|^2\, .
  \end{equation}
  On the other hand, we also have
  \begin{equation*}\label{E:other-bound}
  \alpha = O({|h|})\, .
  \end{equation*}
  To see this, we apply the implicit function theorem (see 
  Appendix II) to write $z$ as above in the form $z = (z', x_n+iy_n)$ with $y_n=\Phi (z', x_n)$ where $\Phi$ is of class $C^{1,1}$ and $\Phi (0', 0)=0$ (recall that $z\in\bndry\D$, so $\rho (z)=0$).

  \noindent{\em Step 3.}\quad 
    First, referring to Proposition \ref{P:1}, 
  if $h\in T^{\C}_w$, we get
  \begin{equation} \label{E:1.1-rho}
  \rho(w+h) =\frac12\inl\nabla^2_{T^{\C}}\rho (w) h, h\inr +o(|h|^2)\quad \mbox{as}\quad h\to 0
  \end{equation}
 for $d\sigma$-almost every $w\in\bndry \D$. This is because $\rho (w)=0$ and $\nabla_T\rho (w) =0$. 
 We now fix a $w$ for which \eqref{E:1.1-rho} holds. Note that
 \begin{equation*}
 \rho (w+h) = -\int\limits_0^\alpha \frac{d}{d t}\left(\rho(w+h+t\nu_w)\right)\, dt
 \end{equation*}
 because $\rho (w+h+\alpha\nu_w) =0$ by Step 2.
 But if $|h|$ is sufficiently small (note that then
  $\alpha$ will also be small), then 
  \begin{equation*}
 \frac{\dee\rho}{\dee\nu_w}(w+h+t\nu_w)= -\,\frac{d}{d t}\left(\rho(w+h+t\nu_w)\right) \geq c_2>0, \ \  \mbox{because}\ \ 
  -\frac{\dee\rho}{\dee\nu_w}(w)\geq c_2'>0
  \end{equation*}
  since $\nabla\rho (w)\neq 0$ for any $w\in\bndry\D$, and $\D=\{\rho<0\}$.
  Thus the integral above shows that 
   $\rho (w+h) \geq c_2\alpha$
     and because of \eqref{E:alpha-bound} we obtain
   \begin{equation*}
   \rho(w+h)\geq c_3|h|^2.
   \end{equation*}
   Combining the latter with \eqref{E:1.1-rho} we have
   \begin{equation}\label{E:conclusion-h-small}
   \inl\nabla^2_{T^{\C}}\rho (w)\, h, h\inr \geq c_4 |h|^2
   \end{equation}
   for all sufficiently small $h\in \Tcw$. Since both sides of \eqref{E:conclusion-h-small} are homogeneous of degree 2, the inequality \eqref{E:conclusion-h-small} extends automatically
   to all $h\in \Tcw$. This proves the conclusion \eqref{E:cplx-tang-hess-positive-def} with $c=c_4$.
  \end{proof}

  We remark that if the domain $\D$ were of class $C^2$ and strongly $\C$-linearly convex, then the positive definiteness given by \eqref{E:cplx-tang-hess-positive-def} would be valid for {\em all} $w\in\bndry\D$: this is implicit in the argument just given. We also note that if we had \eqref{E:conclusion-h-small}, then the condition \eqref{E:Delta-bound} would hold, but only for those $z$ that are sufficiently close to $w$.\\
  
  We next compare strong $\C$-linear convexity to the more standard notion of convexity. If 
  $\D\subset\mathbb R^{2n}$ is a bounded domain of class $C^2$ we say that it is {\em strongly convex} if it is convex and
  \begin{equation}\label{D:strongly-cvx}
  \inl\nabla^2_T\rho (w) H, H\inr\geq c|H|^2\quad \mbox{for any}\ \  H\in T_w\ \  \mbox{and for any}\ \  w\in\bndry\D\, .
  \end{equation}

  One can show that \eqref{D:strongly-cvx} implies
  \begin{equation}\label{E:gradient-bound}
  \inl\nabla\rho (w), w-z\inr\geq c'|w-z|^2\quad \mbox{for any}\ \ z\in \D \ \mbox{and}\ \ w\in\bndry\D\, .
  \end{equation}
     As a consequence, using the above with \eqref{E:pairing-real-im}, \eqref{E:Delta-bound} and \eqref{E:gradient-bound} we see that strong convexty of $\D$ implies strong $\C$-linear convexity.
    
    The converse is not true as can be seen by the following example. Consider the domain $\D$ in $\mathbb C^2$ 
    given
    by 
     $ \rho (z) = |z_1|^2-\mbox{Im}(z_2)$
       which is the half space associated to the Heisenberg group. While $\D$ is unbounded, it is strongly $\C$-linearly convex but not strongly convex. To obtain a bounded variant of $\D$ we may take 
        $ \widetilde\rho (z) = |z_1|^2 -\mbox{Im}(z_2) + |z_2|^4 $
         and note that the (bounded) domain $\D=\{\widetilde\rho (z)<0\}$ is not strongly convex near $w=(0,0)\in\bndry\D$, while it is still strongly $\C$-linearly convex.\\
     
     Next we compare these notions with ``strong pseudo-convexity". Again assuming the bounded domain is of class $C^2$, and choosing a coordinate system $(z', z_n)$ centered at $w\in\bndry\D$ as in Section  \ref{SS:strong C-lin-convex}, the condition of 
     {\em strong pseudo-convexity at $w$} is that the $(n-1)\times(n-1)$ Levi matrix
     \begin{equation*}
     \L =
         \left\{\dee^2_{z_j\overline{z}_k}\rho\, (w)\right\}_{1\leq\, j,\, k\,\leq n-1}
     \end{equation*}
     satisfies
     \begin{equation}\label{E:levi-positive}
     \inl\L\, h, h\inrc\geq c'|h|^2,\quad \mbox{for some}\ c'>0\ \ \mbox{and for any}\ h\in\Tcw\, .
     \end{equation}
          
     We note that \eqref{E:levi-positive} is unchanged under a unitary change of the coordinates 
     $z'\in\mathbb C^{n-1}$, which was the freedom in the choice of coordinates we allowed. If we make an appropriate unitary change of this kind, we may assume that the matrix $\L$ has been
      diagonalized, so that 
\begin{equation*}
           \dee^2_{z_jz_j}\rho\, (w)
      =\lambda_j,\quad \mbox{and}\ \ 
           \dee^2_{z_jz_k}\rho\, (w)
      =0\quad \mbox{if}\ \ j\neq k\, .
      \end{equation*}
      The $\{\lambda_j\}$ are the eigenvalues of $\L$ and the condition \eqref{E:levi-positive} of
      strong pseudo-convexity is then equivalent with $\lambda_j\geq c'$, for $1\leq j\leq n-1$. However
\begin{equation}\label{E:formula-1}
      \lambda_j= 
           \dee^2_{z_jz_j}\rho\, (w) = \frac14
           \left(\dee^2_{x_jx_j}\rho + \dee^2_{y_jy_j}\rho\right)\! (w)\, ,
       \end{equation}
   so if   $\D$ is strongly $\C$-linearly convex it follows from \eqref{E:cplx-tang-hess-positive-def} that
      \begin{equation}\label{E:formula-2}
           \dee^2_{x_jx_j}\rho\,(w)
      \geq c\quad \mbox{and}\quad 
           \dee^2_{y_jy_j}\rho\, (w)
      \geq c\quad \mbox{for}\ \ 1\leq j\leq n-1
      \end{equation}
      and therefore $\lambda_j\geq c/2$. Hence such a $\D$ is 
            strongly pseudo-convex, with
      \eqref{E:levi-positive} holding for $c'=c/2$. One can see that the reverse implication fails. The (unbounded) domain $\D\subset\mathbb C^2$ given by $\rho (z) = 2x_1^2-y_1^2- y_2$
      does not satisfy  condition \eqref{E:Delta-bound}
  at $w_0=(0,0)\in\bndry\D$ because $\Delta(w_0, z) = -z_2/2$ for any $z$, and one may choose
 $z=(y_1, 0)\in\D$ to see that \eqref{E:Delta-bound} fails; on the other hand, $\D$ is easily seen to satisfy 
 \eqref{E:levi-positive} at $w_0$.  
     A bounded domain representing this phenomenon is given by the defining function
           $\widetilde\rho (z) = 2x_1^2 -y_1^2 + |z_1|^4 + |z_2|^2 - y_2$.
           To summarize, we have obtained the following.
      \begin{proposition}\label{P:notions-of-convexity} Suppose $\D\subset \Cn$ is a bounded domain of class $C^2$. 
      Then we have the implications: $\D$ strongly convex $\Rightarrow$ $\D$ strongly $\C$-linearly convex 
      $\Rightarrow$ $\D$ strongly pseudo-convex.
       None of these implications can, in general, be reversed.
      \end{proposition}
      
      \subsection{The quasi-metric $\d$ on $\bndry \D$.}\label{SS:quasi-metric}
      We return to the case when $\D$ is a bounded domain of class $C^{1,1}$ that is 
      strongly $\C$-linearly convex. On $\bndry \D$ we consider
      \begin{equation}\label{E:dist-def}
      \d (w, z) = |\Delta(w, z)|^{\frac12} = |\langle\dee\rho (w), w-z\rangle|^{\frac12}
      \end{equation}
      restricted now to $z$ and $w$ in $\bndry\D$. We show that $\d$ satisfies the properties 
      of a quasi-metric as given in the following.
      \begin{proposition}\label{P:quasi-metric} Suppose that $\D$ is a bounded, strongly $\C$-linearly convex domain of class $C^{1,1}$. Then, for $\d$ defined as above we have
      \begin{itemize}
      \item[(a)] $\d (w, z) = 0\iff w=z$
      \smallskip
      
      \item[(b)] $\d (w, z) \ \approx\ \d(\z, w)$
      \smallskip
      
      \item[(c)] $\d (w, z)\ \lesssim\ \d(w, z') + \d(z', z)$
      \end{itemize}
      for any $w, z, z'\in\bndry\D$.
      \end{proposition}
      Here and in the sequel we use the short-hand $X\lesssim Y$ to mean $X\leq cY$, where the constant $c$ is independent of the variables in question. Also $X\approx Y$ means that $X\lesssim Y$ and $Y\lesssim X$.
      
      \begin{proof} Note that by \eqref{E:Delta-bound} we have $|w-z|\lesssim \d (w, z)$ so we get conclusion (a). 
      
      Next, suppose $\d (w, z) =\delta$, then $|\Delta (w, z)|= \delta^2$ and
      $|w-z|\lesssim \delta$. But $|\Delta (z, w)| = |\langle\dee\rho (z), w-z\rangle|$. Hence
       $||\Delta(w, z)|-|\Delta (z, w)|| \lesssim \delta |z-w|\lesssim \delta^2$, because 
       $|\dee\rho (z) -\dee\rho (w)|\lesssim |z-w|$ since $\rho$ is of class $C^{1,1}$. Thus 
       $|\Delta (z, w)|\lesssim \delta^2$ and (b) is  established. 
  
  To prove (c) it suffices to observe that, by \eqref{E:dist-def},
      \begin{equation*}
    \d^2(w, z)
    =|\langle\dee\rho(w), w-z \rangle|
     \leq |\langle\dee\rho(w) -\dee\rho(z'), z'-z\rangle| + 
     \d^2(z', z) + \d^2(w, z')\, .
     \end{equation*}
     On the other hand we have
    $|\langle\dee\rho(w) -\dee\rho(z'), z'-z\rangle|\leq |\dee\rho (w) - \dee\rho (z')|\,|z'-z|
    \lesssim|w-z'|\,|z'-z|$ and moreover
      $ |w-z'|\,|z'-z|\lesssim \d(w, z')\,\d(z', z)$
      by  \eqref{E:Delta-bound}.
       Combining these inequalities we obtain $\d^2(w, z)\lesssim \big(\d(z', z) + \d(w, z')\big)^{\!2}$,
    proving conclusion (c).
      \end{proof}
         
      A further, useful remark is that
      \begin{equation}\label{E:useful}
      |w - z|\, \lesssim\, \d (w, z)\, \lesssim\, |w-z|^{\frac12}\quad \mbox{for any}\quad w, z\in\bndry\D\, .
      \end{equation}
      
      We have already observed that the left-hand side inequality is due to \eqref{E:Delta-bound},
       the strong $\C$-linear convexity of $\D$. The right-hand side inequality
       follows from the trivial observation that, since $\D$ is bounded, 
       $|\langle\dee\rho (w), w-z\rangle|\lesssim |w-z|$.\\
       
       \subsection{The Leray-Levi measure}\label{SS:Leray-Levi}
      Here again we consider the case when $\D$ is a bounded domain of class $C^{1,1}$ that is 
      strongly $\C$-linearly convex. There is a measure on the boundary of $\D$ that plays a key role in what follows. It is defined via the $(2n-1)$-form 
      \begin{equation*}
      \frac{1}{(2\pi i)^n}\,\dee\rho \wedge(\deebar\dee\rho)^{n-1}\, .
      \end{equation*}
      
      In fact, when $\rho$ is of class $C^{1,1}$, we have shown in Section  \ref{SS:1.8} how to define the $2$-form $j^*(\deebar\dee\rho)$ on $\bndry\D$ which has coefficients that are
      bounded measurable functions. From it, by \eqref{E:(1.12)} we have the linear functional
      
      \begin{equation}\label{E:linear-functional}
      f\ \longmapsto\ \frac{1}{(2\pi i)^n}\int\limits_{w\in\bndry\D}\!\!\!\!\!
      f(w)\, j^*\big(\dee\rho \wedge\!(\deebar\dee\rho)^{n-1})(w) \ =:\ 
      \int\limits_{w\in \bndry\D}\!\!\!\!f(w)\, d\l (w)
      \end{equation}
defined for $f\in C(\bndry\D)$, and this defines the measure $d\l$. Indeed, the
$L^\infty (\bndry\D)$ character of $j^*(\deebar\dee\rho)$ makes \eqref{E:linear-functional} a bounded linear functional on $C(\bndry\D)$, which determines $d\l$. We refer to this measure
as the {\em Leray-Levi measure}.

\begin{proposition}\label{P:Leray-Levi-surface}
If the $C^{1,1}$ domain $\D$ is strongly $\C$-linearly convex, then $d\l$ is equivalent to the induced Lebesgue measure $d\sigma$ on $\bndry\D$ in the following sense. We have
\begin{equation*}
d\l (w) = \La (w)\, d\sigma (w) \quad \mbox{for}\ \sigma\mbox{-a.e.}\ w\in\bndry\D
\end{equation*}
and there are two strictly positive constants $c_1$ and $c_2$ so that
\begin{equation}\label{E:meas-equiv}
c_1\leq \La (w)\leq c_2 \quad \mbox{for}\ \sigma\mbox{-a.e.}\ w\in\bndry\D
\end{equation}
\end{proposition}
\noindent We point out that while the Leray-Levi measure is equivalent to the induced Lebesgue measure in the sense just described, the latter is more intrinsic than the former, in that the induced Lebesgue measure does not depend on the choice of defining function whereas the Leray-Levi measure does. Although not invariant, the Leray-Levi measure is a crucial factor 
in the Cauchy-Leray integral which, as a whole, is invariant, as we shall see in Section \ref{SS:indep-def-funct}.
\begin{proof}
We need the following fact. Suppose $F$ is any real-valued function of class $C^2$ 
defined on $\Cn$. For any $w\in\bndry\D$ denote by ${\tt{det}}(F)(w)$ the determinant of the
$(n-1)\times (n-1)$ Hermitian matrix 
\begin{equation*}
\left\{\frac{\dee^2F (\cdot) }{\dee z_j\dee\overline{z}_k}\right\}_{\!\! w},\quad 1\leq j, k\leq n-1
\end{equation*}
computed in the coordinate system $(z', z_n)$ centered at $w$, which was used above. (Note that by unitary invariance, the value of this determinant does not depend on the particular choice of the special coordinates that are used.) We can then assert that 
\begin{equation}\label{E:det-identity}
j^*\big(\dee \rho \wedge\!(\deebar\dee F)^{n-1})(w) = \gamma_n\, {\tt{det}}(F)(w)|\nabla\rho (w)|\, d\sigma (w),\quad \mbox{for any}\quad w\in\bndry\D\, ,
\end{equation}
where $\gamma_n = (n-1)!/4\pi^n$. This computation is implicit in \cite[Lemma VII.3.9]{R} where
it is given explicitly for $F=\rho$ when $\D=\{\rho<0\}$ is of class $C^2$. To pass to the case when
 $\D$ is of class $C^{1, 1}$, 
we apply \eqref{E:det-identity} to the approximating sequence $\{F_k\}$ given by Proposition \ref{P:2} with $F$ there taken to be $\rho$. Now the left-hand side of \eqref{E:det-identity}, with $F_k$ in place of $F$ clearly converges for $\sigma$-a.e.
 $w\in\bndry\D$ to the left-hand side of \eqref{E:det-identity}, by conclusion {\em (3)} of Proposition \ref{P:2}. Moreover, by conclusion {\em (1)} of that proposition, this convergence is in particular bounded, and hence gives weak convergence of the corresponding sequence of measures. As a result, 
 \begin{equation*}
 d\l = \frac{1}{(2\pi i)^n}\,j^*\!\left(\dee\rho \wedge(\deebar\dee\rho)^{n-1}\right)(w) =
 \gamma_n\,{\tt{det}}(\rho)(w)\cdot|\nabla\rho (w)|\, d\sigma (w)\ \  \mbox{for}\ \ \sigma\mbox{-a.e.}\ w\in\bndry\D.
 \end{equation*}
 In computing ${\tt{det}}(\rho)(w)$, we can further specify our coordinates at $w$ so that
 $\{\dee^2\rho/\dee z_j\dee\overline{z}_k\}$ is diagonal. Recalling \eqref{E:formula-1} and
 \eqref{E:formula-2} and the comments thereafter, we see that
 \begin{equation*}
 {\tt{det}}(\rho)(w)\geq \left(\frac{c}{2}\right)^n
 \end{equation*}
 where $c$ is the constant of $\C$-linear convexity \eqref{E:Delta-bound}. Of course this determinant is clearly bounded above, since all the matrix entries are bounded. This, together 
 with the fact that $|\nabla\rho (w)|\approx 1$ for $w\in\bndry\D$, allows us to conclude the proof of \eqref{E:meas-equiv}.
\end{proof}
The reader may note that this conclusion still holds if $\D$ is of class $C^2$, when one makes the weaker geometric assumption that $\D$ is strongly pseudo-convex. We will return to this point in the forthcoming paper \cite{LS-3}.

\subsection{Some
 estimates}\label{SS:integral-estimates} 
We let $\B_r(w)$ be the ball in $\bndry\D$ given by
\begin{equation*}
\B_r(w) =\{ z\in\bndry \D\ :\  \d (w, z)<r\}
\end{equation*}
where $\D$ is as above.
\begin{proposition}\label{P:surface-balls} 
 We have
\begin{equation*}
\l (\B_r(w))\ \approx\  r^{2n}\quad\  \mbox{for}\quad 0<r\leq 1
\end{equation*}
for any $w\in\bndry\D$.
\end{proposition}
\begin{proof}
By the equivalence in Proposition \ref{P:Leray-Levi-surface}, the conclusion is the same as 
$\sigma((\B_r(w))\approx r^{2n}$, $0<r\leq 1$, with $\sigma$ the induced Lebesgue measure on
$\bndry \D$.
We begin by observing that it suffices to prove the assertion
 when $0<r\leq a$, where $a$ is fixed with $0<a\leq 1$. In fact, when $a\leq r\leq 1$, then
\begin{equation*}
\B_a(w)\subset\B_r(w)\subset\B_1(w)
\end{equation*}
and hence $\sigma (\B_r(w))\approx 1$ while also $r^{2n}\approx 1$. To prove 
$\sigma((\B_r(w))\approx r^{2n}$ when $0<r\leq a$, we need to control the shape of the balls $\B_r(w)$, and this is best done in terms of the coordinate system $(\z', \z_n)$ centered at $w$ we have already used several times. Here the key claim is
\begin{equation}\label{E:shape-balls}
\d(w, z)\approx |\z'| + |x_n|^{1/2} \, ,\quad \mbox{when}\ \ \d(w, z)\leq a_1
\end{equation}
where $a_1$ is to be chosen below, and $\z_n=x_n+iy_n$. In fact $\d^2(w, z) = |\Delta (w, z)|$
and by \eqref{E:Delta-repr}, we have
\begin{equation}\label{E:d-estimate}
\d^2(w, z) = c_w|\z_n|
\end{equation}
where $c_w= |\nabla\rho (w)|/2$. However by \eqref{E:Delta-bound}, $\d(w, z)\gtrsim |z-w|\geq |z'|$, since $w=(0,\ldots, 0)$. Therefore 
\begin{equation}\label{E:d-bound-2}
\d(w, z) \gtrsim |z'| + |z_n|^{1/2} \geq |\z'|+|x_n|^{1/2},
\end{equation}
proving one direction of \eqref{E:shape-balls} (without the limitation that $\d (w, z)\lesssim a_1$).
 In the other direction,
  Taylor's formula applied to the $C^{1,1}$ function $\rho$ gives
 \begin{equation*}
 \rho (\z) = \rho (w) + \inl \nabla\rho (w), \z-w\inr + O(|\z-w|^2)\, .
 \end{equation*}
 Since we are taking $\z, w\in\bndry \D$, then $\rho (\z)=\rho (w) =0$, and in our coordinate system we obtain
 \begin{equation*}
 |\nabla\rho (w)|\,y_n = O(|z-w|^2) = O(|\z'|^2 + |\z_n|^2)\, ,
 \end{equation*}
 because 
 $\dee_{z_j}\rho (w)=0, j=1,\ldots, n-1$, 
  $\dee_{x_n}\rho (w)=0$ 
 and
  $\dee_{y_n}\rho (w)= -|\nabla\rho (w)|$. 
 As a result, 
 \begin{equation*}
 |y_n|\leq c(|\z'|^2 + |\z_n|^2)
 \end{equation*}
 and since $|\z_n|\leq |x_n|+|y_n|$ we have, 
 \begin{equation}\label{E:bound-3}
 |\z_n|\leq c(|\z'|^2 + |x_n| + |\z_n|^2)
 \end{equation}
 once we take the precaution of also choosing $c\geq 1$. It is here that we choose $a_1$, 
 with $a_1^2/c_w=1/2c$ where $c_w$ appears in \eqref{E:d-estimate}, because then
 $\d (w, z)\leq a_1$ implies $c_w|\z_n|\leq a^2_1$. Thus $|\z_n|\leq 1/2c$, and this means that
$c|\z_n|^2<|\z_n|/2$, 
and from this and  \eqref{E:bound-3} we get 
 \begin{equation*}
 |\z_n|\leq 2c(|\z'|^2 + |x_n|)\quad \mbox{when}\quad \d(w, z)\leq a_1.
 \end{equation*}
 Combining this inequality with \eqref{E:d-estimate} establishes \eqref{E:shape-balls}. Finally we apply the implicit function theorem (see 
   Appendix II) to write points $z$ near $w$, with $\rho (z) =0$ (that is $z\in\bndry\D$), in the form $y_n=\Phi (\z', x_n)$. This can be done for points $\z$ that lie above 
 $(\z', x_n)\in\mathbb C^{n-1}\times\mathbb R$, for $\d (\z, w) \leq a_2$ for an appropriately small $a_2$,
 and this implies that $|\z'|^2 + |x_n|$ is small with $a_2$, in view of \eqref{E:d-bound-2}. In this situation, as is well-known, we can express the induced Lebesgue measure $d\sigma$ on the graph 
 $\{y_n=\Phi (\z', x_n)\}$ as
 \begin{equation*}
 d\sigma  = A (z', x_n) \,d\m 
 \end{equation*}
where $\displaystyle{d\m  =\left(\prod\limits_{j=1}^{n-1}dx_j\,dy_j\right)\!dx_n}$ is Euclidean measure on $\mathbb C^{n-1}\times \mathbb R$, and

\begin{equation*}
A(\z', z_n) = \frac{|\nabla\rho (\z', \z_n)|}
{\left|\dee_{y_n}\rho\, (z', z_n)\right|}\, .
\end{equation*}

Also
\begin{equation}\label{E:surf-meas-estimate}
A(\z', x_n)\approx 1\, ,
\end{equation}

if $a_2$ is small enough. So we choose $a=\min\{a_1, a_2\}$, and therefore both 
\eqref{E:shape-balls} and \eqref{E:surf-meas-estimate} are valid when $\d(w, z)\leq a$. 
Now clearly, 
$d\m\left(\{|\z'|+|x_n|^{1/2}\leq r\}\right)\approx r^{2n}$, 
  so that $\sigma((\B_r(w))\approx r^{2n}$, 
 whenever $0<r\leq a$. The proposition is therefore proved.
\end{proof}

It is important to observe that the constants that arise in the proposition above, in particular the two constants implicit in the assertion $\l (\B_r(w))\approx r^{2n}$, the constant $a$, etc., can be taken to be independent of $w\in\bndry\D$. This is because, as the reader may verify, all these constants are controlled by two quantities, namely 
\begin{equation*}
M=\|\rho\|_{C^{1,1}},\quad \mbox{and}\quad m=\min\limits_{w\in\bndry\D}|\nabla\rho(w)|.
\end{equation*}

A simple but useful consequence of Proposition \ref{P:surface-balls} is the following set of integral estimates.
\begin{corollary}\label {C:int-ests} With $\D$ as above,
 for any $w\in\bndry\D$ and for $\epsilon>0$, we have

\begin{equation}\label{E:int-est-a}
\int\limits_{z\in \B_r(w)}\!\!\!\!\! \d(w, z)^{-2n +\epsilon}\, d\lambda(\z)\ \leq\ c_\epsilon\, r^\epsilon,\quad \mbox{for}\ \ 0< r\leq 1
\end{equation}

\begin{equation}\label{E:int-est-b}
\int\limits_{z\in \B_r(w)^c}\!\!\!\!\! \d(w, z)^{-2n -\epsilon}\, d\lambda(\z)\ \leq\ c_\epsilon\, r^{-\epsilon},\quad \mbox{for}\ \ 0< r\leq 1
\end{equation}
\begin{equation}\label{E:int-est-c}
\int\limits_{z\in \B_r(w)^c}\!\!\!\!\! \d(w, z)^{-2n}\, d\lambda(\z)\ \leq\ c\, \log(1/r),\quad \mbox{for}\ \ 0< r\leq 1/2
\end{equation}
\end{corollary}
\begin{proof}
For the proof of \eqref{E:int-est-a}, write $\B_r(w)$ as a disjoint union
 \begin{equation*}
 \B_r(w) =\bigcup\limits_{k=1}^\infty\bigg(\B_{r2^{-n}}(w)\setminus \B_{r2{-n-1}}(w)\bigg)
 \end{equation*}
 Then the integral is  majorized by  a multiple of
 \begin{equation*}
 \sum\limits_{k=0}^\infty \,(r2^{-k-1})^{-2n +\epsilon}\,\l (\B_{r2^{-n}}(w))\
 \lesssim\ r^{-2n +\epsilon}r^{2n}\sum\limits_{k=0}^\infty\, 2^{-k(-2n+\epsilon)}\cdot 2^{-2kn}
  \end{equation*}
 \begin{equation*}
=\ \, r^\epsilon\sum\limits_{k=0}^\infty \,2^{-k\epsilon} \ =\  c_\epsilon r^\epsilon\, 
 \end{equation*}
 The proofs of \eqref{E:int-est-b} and \eqref{E:int-est-c} are similar and may be left to the reader.
\end{proof}

\section{The Cauchy-Leray integral}\label{S:CLI}
We recall that we are dealing with a bounded domain $\D$ in $\Cn$, with a defining function 
$\rho$ that is of class $C^{1,1}$. We also suppose that $\D$ is strongly $\C$-linearly convex, which means that $\Delta (w, z) =\langle\dee\rho (w), w-z\rangle$ satisfies the two equivalent conditions \eqref{E:Delta-bound} and \eqref{E:dist-bound}. 

The {\em Cauchy-Leray integral} of a suitable function $f$ on $\bndry\D$, denoted $\CI (f)$, is formally defined by
\begin{equation}\label{E:CLI-def-0}
\CI (f) (z) =\frac{1}{(2\pi i)^n}\int\limits_{w\in\bndry\D}\!\!\!\frac{f(w)}{\Delta (w, z)^n}
\dee\rho (w) \wedge (\deebar\dee\rho (w))^{n-1},\quad z\in \D\, .
\end{equation}
The actual definition is as follows. Recall the Leray-Levi measure
\begin{equation*}
d\l(w) = j^*\left(\frac{\dee\rho (w) \wedge (\deebar\dee\rho (w))^{n-1}}{(2\pi i)^n}\right) =
\Lambda (w) d\sigma (w),\quad w\in\bndry\D
\end{equation*}
given in Section \ref{SS:Leray-Levi} via Proposition \ref{P:Leray-Levi-surface}, with $d\sigma$ the induced Lebesgue measure on $\bndry\D$ and $\Lambda (w)\approx 1$, see 
\eqref{E:meas-equiv}. Then whenever $f$ is integrable with respect to $d\sigma$ (or what is the same, with respect to $d\l$), the precise form of \eqref{E:CLI-def-0} is the function $\CI (f)$, defined for $z\in\D$ by
\begin{equation}\label{E:CLI-def}
\CI (f) (z) = \int\limits_{w\in\bndry\D}\!\!\! \frac{f(w)\, d\l (w)}{\Delta (w, z)^n}\, .
\end{equation}
The purpose of this section is to prove two basic propositions about the Cauchy-Leray integral \eqref{E:CLI-def}. The first shows that despite the fact that the definition \eqref{E:CLI-def} seems to depend on the particular choice of the defining function $\rho$ 
that is used, the integral is in fact independent of the choice of $\rho$, and so it depends only on the domain $\D$. The second proposition shows that in addition to the fact that the integral \eqref{E:CLI-def} always produces holomorphic functions (because $\Delta (w, z)$ is holomorphic in $z\in\D$ for each fixed $w\in\bndry \D$), in fact it reproduces holomorphic functions. Hence it has every right to
be called a ``Cauchy integral''.

\subsection{Independence of the choice of defining function.}\label{SS:indep-def-funct}
Suppose $\rho'$ is another $C^{1,1}$ defining function of the domain $\D$, and write 
$d\l '(w) = j^*\left(\dee\rho' (w) \wedge (\deebar\dee\rho' (w))^{n-1}\right)\!\!/(2\pi i)^n$,
$\Delta'(w, z) = \langle \dee\rho'(w), w-z\rangle$.
\begin{proposition}\label{P:indep-def-funct} We have that
\begin{equation}\label{E:indep-def-funct}
\frac{d\l '(w)}{\Delta' (w, z)^n} = 
\frac{d\l (w)}{\Delta (w, z)^n}
\end{equation}
\end{proposition}
This shows that $\CI$ 
is independent of the choice of defining function.

To prove the proposition, consider first the special case when $\D$ is of class $C^2$, and the connection between $\rho$ and $\rho'$ is given by
$\rho' = a\,\rho$, where $a$ is also of class $C^2$. Then it is easily verified that $j^*(\dee\rho') = a j^*(\dee\rho)$, $\Delta'(w, z) = a \Delta (w, z)$, and 
$$
j^*(\dee\rho)\wedge j^*(\deebar\dee a\rho)^{n-1} = a^{n-1} j^*(\dee\rho)\wedge j^*(\deebar\dee\rho)^{n-1}
$$
from which \eqref{E:indep-def-funct} follows immediately. The general case (when $\rho$ and
$\rho'$ are merely of class $C^{1,1}$) is more subtle.
We begin the proof of 
\eqref{E:indep-def-funct}
by noting first that $\rho' =a \rho$, where $a$ is a Lipschitz function, with $a>0$ near $\bndry\D$. This follows from Remark 2 in Appendix II. Next we observe that 
\begin{equation}\label{E:deriv-lip}
j^*(\dee\rho') = a\, j^*(\dee\rho),
\end{equation}
where both sides of \eqref{E:deriv-lip} are forms with Lipschitz coefficients. In fact
\begin{equation*}
\nabla\rho' = \nabla (a\rho) = a\,\nabla\rho + \rho\,\nabla a
\end{equation*}
where the components of $\nabla a$ (the first-order derivatives of $a$) are taken in the sense of distributions, and are $L^\infty$ functions of the underlying $\mathbb R^{2n}$, and thus are defined and bounded almost everywhere on $\mathbb R^{2n}$. So there is a subset $\D_0\subset\D$, so that $\D\setminus\D_0$ has $2n$-dimensional measure zero, with $|\nabla a (z)|\leq M$, for every $z\in\D_0$. As a result
\begin{equation}\label{E:diff-quot-lip}
|\nabla\rho'(z)-a(z)\nabla\rho (z)|\leq M|\rho (z)|,\quad \mbox{if}\ z\in\D_0.
\end{equation}
Since $\D_0$ is dense in $\overline\D$, then for each $w\in\bndry\D$ there is a sequence $\{z_k\}\subset \D_0$, with $z_k\to w$. So by continuity of the left-hand side of \eqref{E:diff-quot-lip} it
 follows that $(\nabla\rho')(w) = a(w)\,\nabla\rho (w)$, if $w\in\bndry\D$ and hence in particular \eqref{E:deriv-lip} is established. This also shows that
 \begin{equation}\label{E:denom-indep}
 \Delta'(w,z) = 
  a(w)\,\Delta (w,z)\, 
 \end{equation}
 The rest of the proof of \eqref{E:indep-def-funct} depends on suitable approximations that rely on the results of Section \ref{S:C11}. We take $\{\rho_k\}$ to be the approximation sequence of the
 $C^{1,1}$ function $F:= \rho$ given by Proposition \ref{P:2}.
   Similarly, 
 $\{a_k\}$ will be the approximation sequence of the Lipschitz function $G:=a$ given by Proposition \ref{P:approx-LF}. 
  These approximations will be applied via
 the two lemmas below.
 \begin{lemma}\label{L:indep-def-funct-1} The sequence $\{j^*(\deebar\dee(a_k\rho_k))\}$ is
 \begin{enumerate}
      \item[(a)\ ] uniformly bounded, and
      \item[(b)\ ] converging almost everywhere to $j^*(\deebar\dee(\rho'))$.
      \end{enumerate}
 \end{lemma}
 To prove this lemma, we recall that the sequences
  $\|\rho_k\|,\ \|\nabla\rho_k\|,\ \|\nabla^2\rho_k\|$, as well as $\|a_k\|,\ \|\nabla a_k\|,\ 
  \|\nabla^2 a_k\|/k$ are all bounded, where $\|\cdot\|$ denotes the sup-norm over $\mathbb R^{2n}$ (see item {\em (1)} in Proposition \ref{P:2}, and item {\em (1)}
  in Proposition \ref{P:approx-LF}, respectively). Also $\{a_k\}$ and $\{\nabla\rho_k\}$ converge uniformly on $\bndry\D$ to $a$ and $\nabla\rho$, respectively. Moreover, 
  $\rho_k(w) =\rho_k(w)-\rho(w)= O(1/k^2)$, and $(\nabla\rho_k)(w) -(\nabla\rho)(w) = O(1/k)$, if $w\in\bndry\D$, by \eqref{E:unif-quantified} in Proposition \ref{P:2}.
Finally, $j^*(d a_k)(w)$ converges for almost every $w\in\bndry\D$ (see item {\em (3)} in Proposition \ref{P:approx-LF}). Now $j^*(\deebar\dee(a_k\rho_k))= j^*(d\,\dee(a_k\rho_k))$ and the latter equals
\begin{equation}\label{E:approx-LF-aux-1}
j^*\bigg(
(d\,\dee a_k)\rho_k - \dee a_k\wedge d\rho_k + 
d a_k\wedge \dee\rho_k + (d\,\dee\rho_k)a_k
\bigg).
\end{equation}

\noindent The first term in the expression above converges uniformly to zero since its norm is 

\noindent $O(k\cdot 1/k^2)$. For the second term note that $j^*(d\rho)=0$, and thus 
$j^*(d\rho_k)= O(1/k)$, so this term tends uniformly to zero. Next,
observe that the sequence
$j^*(d a_k)$ is bounded and it converges almost everywhere, and the same is true for $j^*(\dee\rho_k)$, so the third term converges almost everywhere and is bounded. Finally, the sequence
$j^*\big((d\,\dee\rho_k)a_k\big)$ converges almost everywhere by Proposition \ref{P:3} and is bounded.

Hence we see that
$\{j^*\big(\deebar\dee(a_k\rho_k)\big)\}$ converges almost everywhere and boundedly to a limit which we denote $L$. It only remains to verify that 
$L= j^*(\deebar\dee (\rho'))$, or what is the same, that
\begin{equation}\label{E:approx-LF-aux-2}
\int\limits_{w\in\bndry\D}\!\!\!\!
L(w)\wedge \psi (w) \ =\ 
\int\limits_{w\in\bndry\D}\!\!
\bigg(j^*(\deebar\dee (\rho'))(w)\bigg)
\wedge \psi (w)
\end{equation}
for every $C^1$ form $\psi$ on $\bndry\D$ that has degree $2n-3$. Now according to \eqref{E:(1.11)}, the right-hand side of \eqref{E:approx-LF-aux-2} equals 
$$
\int\limits_{w\in\bndry\D}\!\!\!
j^*(\dee (\rho'))(w)
\wedge d\psi (w),
$$
which, in turn, equals $\int_{\bndry\D} a\,j^*(\dee\rho)\wedge d\psi$, by \eqref{E:deriv-lip}. On the other hand, the left-hand side of \eqref{E:approx-LF-aux-2} is the limit as $k\to\infty$ of 
$\int_{\bndry\D} j^*\big(\deebar\dee(a_k\rho_k)\big)\wedge \psi$; but for each $k$ the latter equals 
$\int_{\bndry\D} j^*\big(\dee(a_k\rho_k)\big)\wedge d\psi$, again by \eqref{E:(1.11)}. This concludes the proof of 
\eqref{E:approx-LF-aux-2} and thus of Lemma \ref{L:indep-def-funct-1}.
As a consequence,
\begin{equation*}\label{E:approx-LF-aux-3} 
j^*(\deebar\dee(a_k\rho_k))
^{n-1}\to\
j^*(\deebar\dee(\rho'))
^{n-1}\qquad \mbox{almost everywhere.}
\end{equation*}

\begin{lemma}\label{L:indep-def-funct-2} We have that,
\begin{equation*}
j^*(\dee\rho)\wedge
j^*(\deebar\dee(a_k\rho_k))^{n-1}\to\
a^{n-1}j^*(\dee\rho)\wedge 
j^*(\deebar\dee(\rho))
^{n-1}\quad \mbox{a.e.}\ \ \bndry\D\, .
\end{equation*}
\end{lemma}
To prove this lemma, 
observe
that 
$j^*(\deebar\dee(a_k\rho_k))^{n-1}$ is a sum of terms, each a wedge product of factors taken from the four terms appearing in \eqref{E:approx-LF-aux-1}. These can be dealt with as before, by 
further noting that
 $$
 \dee\rho\wedge\dee\rho_k = \dee\rho \wedge (\dee\rho_k-\dee\rho)=O(1/k).
 $$
Concerning
 the term 
 $j^*(\dee\rho)\wedge (j^*(d\,\dee\rho_k)a_k)^{n-1} = 
 a_k^{n-1}j^*(\dee\rho)\wedge j^*(d\,\dee\rho_k)^{n-1}$, we note that this converges almost everywhere to
 $a^{n-1}j^*(\dee\rho)\wedge j^*(\deebar\dee\rho)^{n-1}$ by \eqref{E:(1.12)} and the comments thereafter, thus concluding the proof of Lemma \ref{L:indep-def-funct-2}. 
 
 We now combine Lemmas \ref{L:indep-def-funct-1} and \ref{L:indep-def-funct-2}, together with 
 \eqref{E:deriv-lip} and \eqref{E:denom-indep}, to see that \eqref{E:indep-def-funct} has been 
 verified, therefore establishing Proposition \ref{P:indep-def-funct}.

  \subsection{Holomorphic character and reproducing property}\label{SS:repr-hol-fncts} We now  establishes the further defining properties of \eqref{E:CLI-def}  that will lead to 
  the corresponding operator on the boundary, the Cauchy-Leray transform, see \eqref{E:limit-a} and 
  \eqref{E:CI-CT}.  
  
 \begin{proposition}\label{P:repr-hol-fncts}
 \qquad\qquad\qquad\qquad\qquad\qquad
 \begin{enumerate}
 \item[(a)\ ] Whenever $f$ is an integrable function on $\bndry\D$, the Cauchy-Leray integral 
 $\CI (f)$ given by \eqref{E:CLI-def} is holomorphic in $\D$.
 \item[(b)\ ] Suppose $F_0$ is continuous in $\overline{\D}$ and holomorphic in $\D$, and let
 $$
 f_0 = F_0\bigg|_{\bndry\D}
 $$
 Then $\CI (f_0)(z)=F_0(z)$ for any $z$ in $\D$.
  \item[ (c)\ ] 
  Suppose that the function $f$ satisfies the
  H\"older-like condition
 \begin{equation}\label{E:holder-d}
 |f(w_1)-f(w_2)|\lesssim \d (w_1, w_2)^\alpha,\quad \mbox{for all}\ w_1, w_2\in\bndry\D.
 \end{equation}
 Then $\CI (f)$ extends continuously to $\overline{\D}$.
 \end{enumerate}
 \end{proposition}
 Here and in the sequel we will restrict ourselves to $0<\alpha<1$. We note for future reference 
 that the class of functions satisfying condition \eqref{E:holder-d}
 includes the class $C^1(\bndry\D)$.

The first conclusion is self-evident since $\Delta (w, z)$ is linear in $z$ and $|\Delta (w, z)|>0$ whenever $z\in \D$ and $w\in\bndry\D$ by \eqref{E:Delta-bound}. 
The proof of conclusion {\em (b)} is based on the Cauchy-Fantappi\` e formalism. The version we need can be found in \cite{LS-2}; other versions are in \cite[Chapter IV]{R}. We assume that $\D$ is a bounded $C^1$ domain and $\eta$ is a ``$C^1$ generating form''. This means that $\eta = \eta (w, z)$ is a form of type $(1, 0)$ in $w$, with coefficients that are $C^1$ functions which satisfy the following property: for any $z\in\D$ there is
a neighborhood $\U$ of $\bndry\D$ so that 
$\U\cap \{z\}=\emptyset$
 and $\langle\eta(w, z), w-z\rangle =1$, whenever
  $w\in \U$. Under these assumptions, if $F_0\in C^1(\overline{\D})$ and is holomorphic in $\D$, we have \cite[Proposition 2]{LS-2}
\begin{equation}\label{E:gener-repr-holo}
F_0(z) = \frac{1}{(2\pi i)^n}\int\limits_{w\in\bndry\D}\!\!\!\! f_0(w)\, j^*\big(\eta\wedge(\deebar_w\eta)^{n-1}
(w, z)\big),\quad z\in\D\, .
\end{equation}
Now we would like to apply this formula to the case when 
\begin{equation}\label{E:gen-def-form}
\eta(w, z)=\frac{\dee\rho (w)}{\Delta (w, z)}\, .
\end{equation}
Note that $\eta$ is well defined because we know that $\Delta (w,z)\neq 0$ whenever $w\in\bndry\D$ and $z\in\D$ (by the strong $\C$-linear convexity of $\D$); by continuity, it follows that 
for any $z\in\D$ there is an open neighborohood $\U$
of $\bndry\D$ such that $\U\cap \{z\}=\emptyset$ and
$\Delta (w,z)\neq 0$
whenever $w\in \U$, 
 so that $\eta$ is a generating form. Inserting this choice of $\eta$
 in the righthand side of \eqref{E:gener-repr-holo} would give precisely $\CI(f_0)(z)$, the Cauchy-Leray integral of $f_0$; however, since we are dealing with domains that are only of class $C^{1,1}$, the difficulty is that in general $\eta$ as given by 
\eqref{E:gen-def-form} is not of class $C^1$ and so the Cauchy-Fantappi\`e formula \eqref{E:gener-repr-holo} cannot be applied to such $\eta$.
To get around this obstacle we use the approximation by smooth functions $\{\rho_k\}$ that was obtained earlier, 
 see e.g., Lemma \ref{L:indep-def-funct-1}. If we set 
\begin{equation*}
\Delta_k(w, z) =\langle\dee\rho_k(w), w-z\rangle
\end{equation*}
and use the fact that $\nabla\rho_k$ converges to $\nabla\rho$ uniformly on $\bndry\D$ (Proposition \ref{P:2}) and that $\nabla^2\rho_k$ are bounded uniformly in $w\in\bndry\D$ and 
in $k\geq k(z)$ (Proposition \ref{P:approx-LF}),
we obtain that $\Delta_k(w, z)\neq 0$ for any $w\in \U$ and from this
 we can conclude that 
\begin{equation*}
\eta_k(w, z) = \frac{\dee\rho_k(w)}{\Delta_k(w, z)}
\end{equation*}
is a $C^1$ generating form for any $k\geq k(z)$. It now follows 
 that \eqref{E:gener-repr-holo} holds with $\eta$ replaced by $\eta_k$. 
 However
 \begin{equation*}
 \deebar\eta_k=\frac{\deebar\dee\rho_k}{\Delta_k}\ +\
 \dee\rho_k\wedge\deebar\left(\frac{1}{\Delta_k}\right)
 \end{equation*}
 and since $\dee\rho_k\wedge\dee\rho_k=0$ we obtain
 \begin{equation*}
 \eta_k\wedge (\deebar\eta_k)^{n-1} = \frac{\dee\rho_k\wedge(\deebar\dee\rho_k)^{n-1}}{\Delta_k^n}\, .
  \end{equation*}
 Therefore the righthand side of \eqref{E:gener-repr-holo} becomes
 \begin{equation*}
 \frac{1}{(2\pi i)^n}
 \int\limits_{w\in\bndry\D}\!\!\!\! f_0(w)\, 
 \frac{j^*\big(\dee\rho_k(w)\big)\wedge j^*\big(\deebar\dee\rho_k)^{n-1}(w)\big)}
 {\Delta_k(w, z)^n}\, .
  \end{equation*}
By the uniform convergence of $\Delta_k(w, z)^{-n}$ to $\Delta (w, z)^{-n}$ (recall that $\Delta(w, z)\neq 0$) and of 
$j^*\!\big(\dee\rho_k\big)$ to $j^*\!\big(\dee\rho\big)$ (Proposition \ref{P:2}), as well as the bounded pointwise convergence of $j^*\!\big(\deebar\dee\rho_k\big)$ to $j^*\!\big(\deebar\dee\rho\big)$
(Proposition \ref{P:3}), we see that the above integral converges to
\begin{equation*}
  \int\limits_{w\in\bndry\D}\!\!\!\! \frac{f_0(w)\,d\l (w)} 
 {\Delta (w, z)^n}\, 
 \end{equation*}
once we recall the definition of the Leray-Levi measure $d\l$ given by 
\eqref{E:linear-functional}. Conclusion {\em (b)} is thus established.
\medskip

We turn to the proof of conclusion {\em (c)}. We begin by applying conclusion 
{\em (b)} in the special case: 
$F_0(z)\equiv 1$, 
  giving us that
\begin{equation}\label{E:CL-repr-spec}
\int\limits_{w\in\bndry\D}\!\!\!\! \frac{d\l (w)} 
 {\Delta (w, z)^n}\ =\ 1\quad \mbox{for all}\ z\in\D\, .
\end{equation}
 Now for any $z\in\D$ denote by $\dz$ the distance of $z$ from $\bndry\D$. Then for some $c>0$
 that is sufficiently small, every $z\in\D$ with $\dz \leq c$ has an orthogonal projection 
 $$
z\mapsto \pz \in\bndry\D
 $$
 for which $|z-\pz|=\dz$, and $z=\pz +\dz\nu_{\pz},$ where $\nu_{\pz}$  is the inner unit normal vector 
 at $\pz\in\bndry\D$. Hence we can write
 \begin{equation}\label{E:CL-diff-repr}
 \CI (f) (z) - f(\pz) = \int\limits_{w\in\bndry\D}\!\!\!
 \frac{f(w)-f(\pz)}{\Delta (w, z)^n}\, d\l (w)\, .
 \end{equation}
 To estimate this integral we need the following lemma, which we will use several times: 
 
 \begin{lemma}\label {L:Delta-epsilon}
        Let $\D=\{\rho<0\}$ 
        be a bounded, strongly $\C$-linearly convex domain
  of class $C^{1,1}$. 
   Then, for sufficiently small $\epsilon$ and for any $\zp, w\in\bndry\D$
      we have
\begin{equation*}
|\Delta  (w, \zp)| +\epsilon\  \approx\  |\Delta (w,  z_\epsilon)|
\end{equation*}
 where $ z_\epsilon = \zp+\epsilon\nu_{\zp}$, with $\nu_{\zp}$ the  inner unit normal vector at
  $\zp$.
\end{lemma}

\noindent{\em Proof of Lemma \ref {L:Delta-epsilon}.\ }
Fix $w\in \bndry \D$. 
We distinguish two cases: $\zp$ close to $w$, and $\zp$ far from $w$.
In dealing with the first case,
by the implicit function theorem (Remark 2 in Appendix II)
we may assume that in the coordinate system (centered at $w$) that was introduced earlier, 
\begin{equation*}
\rho (z) = a(z)(\Phi (z', x_n) - y_n)\, ,\quad z=(z', z_n)\quad \mbox{and}\quad z_n=x_n+iy_n
\end{equation*}
whenever $z\in\overline
D$ is close to $w$, where $\Phi$ is a $C^{1, 1}$-smooth function and $a$ is strictly positive and Lipschitz. We then observe that by
\eqref{E:denom-indep}, we may take 
 without loss of generality 
   $\rho (z) = \Phi (z', x_n) - y_n$, when $z$ is sufficiently close to $w$.
We claim that 
\begin{equation}
\label{E:fact-1'}
\Phi (z', 0)> 0 \quad \mbox{whenever}\quad z'\neq 0 \, .
\end{equation}
Indeed, if $\Phi(z', 0)\leq 0$, then the point $(z', 0)$ would be in $\overline{\D}$; but 
in our coordinate system, see \eqref{E:Delta-repr}, $\Delta (w, (z', 0))=0$, and this contradicts \eqref{E:Delta-bound}, as 
 $|w-(z', 0)|= |z'|>0$. Note that 
  it follows by continuity from \eqref{E:fact-1'} 
 that
$ \Phi (z', 0)\geq 0 \ \mbox{for any} \ z'$.
 Now $\Delta(w,  z_\epsilon) = \Delta  (w, \zp) -\epsilon\langle\dee\rho (w), \nu_{\zp}\rangle$, with 
  $ z_\epsilon = \zp+\epsilon \nu_{\zp}$. However $\langle\dee\rho (w), \nu_w\rangle = -1/2$, since
 $\dee\rho (w) = (0, i/2)$ and $\nu_w = (0, i)$. 
 Moreover, since $\D\in C^{1,1}$ we have
 \begin{equation}\label{E:normals}
 \nu_{\zp} = \nu_w + O(|\zp-w|)
 \end{equation}
 Hence 
 \begin{equation}\label{E:Delta-Delta-eps}
 \Delta (w,  z_\epsilon) = \Delta  (w, \zp) +\epsilon/2 + O(\epsilon |\zp-w|)\, ,
 \end{equation}
 and since $\zp \in\bndry\D$, we have
 \begin{equation}\label{E:fact-0}
 \Delta  (w, \zp) = \frac12(y_n -ix_n) = \frac12(\Phi (z', x_n) -i x_n)
 \end{equation} 
which gives
 $ |\Delta  (w, \zp)|\approx |\Phi (z', x_n)| + |x_n|$.
However
$ \Phi (z', x_n) = \Phi (z', 0) + O(|x_n|)$,
 since $\Phi$ is of class $C^1$, and clearly $|X| + O(|x_n|) \approx \big|X+ O(|x_n|)\big| + O(|x_n|)$; taking
$ X = \Phi (\zp', 0)\geq 0$
  we see that
 $|\Delta  (w, \zp)|\ \approx\ \Phi (z', 0) + O(|x_n|)$.

 Next, using \eqref{E:fact-0}, we find
  \begin{equation*}
 \Delta (w,  z_\epsilon) = \frac12\Phi (\zp', 0) + O(|\xp_n|)+\frac{\epsilon}{2} -\frac i2 x_n + O(\epsilon |\zp-w|)
 \end{equation*}
 and this grants
  \begin{equation}\label{E:ineq-zero}
| \Delta (w,  z_\epsilon)| \approx \bigg|\frac12\Phi (\zp', 0) + O(|x_n|) + \epsilon/2  + 
O(\epsilon |\zp-w|)\bigg|\ +\ 
\bigg|\frac 12 x_n + O(\epsilon |\zp-w|)\bigg|\,.
 \end{equation}
We now write, for convenience: $|O(|x_n|)|\leq A|x_n|$ and we choose $\zp\in\bndry\D$ sufficiently close to $w$ so that
$|O(\epsilon |\zp-w|)|\leq \min\{\epsilon/6,\ \epsilon/12A\}$.

\noindent We distinguish two sub-cases: when $A|x_n|\leq \epsilon/6$, and when 
$A|x_n|> \epsilon/6$. 

If $A|x_n|\leq \epsilon/6$, we have that each of $O(|x_n|)$ and
 $O(\epsilon |\zp-w|)$ is majorized by $\epsilon/6$, and from this it follows that 
$$
| \Delta (w,  z_\epsilon)| \gtrsim \bigg|\frac12\Phi (\zp', 0) + O(|x_n|) + \epsilon/2  +
 O(\epsilon |\zp-w|)\bigg|\gtrsim \Phi (z', 0) +\epsilon\, .
$$
Since
 $A|x_n|\leq \epsilon/6$ then it follows from the above that 
$$
| \Delta (w,  z_\epsilon)|\ \gtrsim\ \Phi (z', 0) + |x_n| + \epsilon\ \approx\ |\Delta (\zp, w)|+\epsilon\, .
$$
On the other hand, when $A|x_n|> \epsilon/6$, then it follows  from $|O(\epsilon |\zp-w|)|\leq \epsilon/12 A$ 
 that 
\begin{equation*}
 \bigg|\frac 12 x_n + O(\epsilon |\zp-w|)\bigg|\gtrsim |x_n|/2
\end{equation*}
so that 
\begin{equation*}
| \Delta (w,  z_\epsilon)| \gtrsim \bigg|\frac12\Phi (\zp', 0) + O(|x_n|) + \epsilon/2  + 
O(\epsilon |\zp-w|)\bigg| + |O(|x_n|)|.
\end{equation*}
From this we obtain
\begin{equation*}
| \Delta (w,  z_\epsilon)| \gtrsim \bigg|\Phi (\zp', 0) + \epsilon  + O(\epsilon |\zp-w|)\bigg| + 
|O(|x_n|)|\gtrsim
\Phi (\zp', 0) + \epsilon + |O(|x_n|)|\approx |\Delta (\zp, w)|+\epsilon.
\end{equation*}
The opposite inequality: $| \Delta (w,  z_\epsilon)| \lesssim |\Delta (w, \zp)|+\epsilon$ is trivial because $\Delta (w,  z_\epsilon) = \Delta  (w, \zp) + O(\epsilon)$, see \eqref{E:Delta-Delta-eps},
 (recall that $\D$ is bounded).
 This concludes the proof of the case when $\zp = (z', z_n)\in\bndry \D$ is close to $w$;
 finally, the proof for the case when $\zp\in\bndry\D$ is away from $w$ is trivial, because then $|\Delta  (w, \zp)|\approx 1$, by strong $\C$-linear convexity, and clearly $|\Delta (w,  z_\epsilon) - \Delta  (w, \zp)|\lesssim \epsilon$, so if $\epsilon$ is small then $|\Delta (w,  z_\epsilon)|\approx 1$, which proves the lemma.
 \qed

 \smallskip
 
 We resume the proof of part {\em (c)} of Proposition \ref{P:repr-hol-fncts}. Setting 
 $ z_\epsilon = \zp+\epsilon\nu_{\zp}$ with $\zp\in\bndry\D$ as in the lemma just proved, we can rewrite
  \eqref{E:CL-diff-repr} as
 \begin{equation*}
  \CI (f) (z_\epsilon) = \int\limits_{w\in\bndry\D}\!\!\!
 \frac{f(w)-f(\zp)}{\Delta (w, z_\epsilon)^n}\, d\l (w) + f(\zp), \quad \zp\in\bndry\D.
 \end{equation*}
 We note first that $\lim\limits_{\epsilon\to 0} \CI (f)(z_\epsilon)$ exists for each fixed $\zp\in\bndry\D$ and we designate this limit as $\CT f(\zp)$. In fact, according to Lemma \ref{L:Delta-epsilon}, $|\Delta (w, z_\epsilon)|\gtrsim |\Delta(w, z)|=\d(w, z)^2$. Moreover, since $f$ satisfies \eqref{E:holder-d}, the integrand above is  majorized by a multiple of \, $\d(w, z)^{-2n+\alpha}$, which is an integrable function in view of \eqref{E:int-est-a}.
 The dominated convergence theorem then guarantees the existence of the limit and we have
 \begin{equation}\label{E:limit-a}
 \CT (f)(\zp) = \int\limits_{w\in\bndry\D}\!\!\!
 \frac{f(w)-f(\zp)}{\Delta  (w, \zp)^n}\, d\l (w) + f(\zp), \quad \zp\in\bndry\D.
 \end{equation}
  This defines the {\em Cauchy-Leray transform}, $f\mapsto\CT(f)$. 
  
  Note that if $\z$ lies outside the support of $f$, the above formula takes the simpler form
  
  \begin{equation}\label{E:limit-a-out}
 \CT (f)(\zp) = \int\limits_{w\in\bndry\D}\!\!\!
 \frac{f(w)}{\Delta  (w, \zp)^n}\, d\l (w)\, .
 \end{equation}

 We next see that the above convergence is in fact uniform, and more precisely, that
 \begin{equation}\label{E:unif-conv-a}
 \sup\limits_{\zp\in\bndry\D}|\CI(f)(z_\epsilon)-\CT(f)(\zp)|\lesssim \epsilon^{\alpha/2}.
 \end{equation}
 Indeed,
 \begin{equation*}
 \CI(f)(z_\epsilon)-\CT(f)(\zp) =  \int\limits_{w\in\bndry\D}\!\!\!
 \left[\frac{1}{\Delta (w, z_\epsilon)^n}- \frac{1}{\Delta  (w, \zp)^n}\right](f(w)-f(\zp))\, d\l(w)\, .
 \end{equation*}
 We break the integral into two parts, $I+II$, according to whether the integration above is taken over those $w$ for which $\d (w, z)\leq \epsilon^{1/2}$, or $w$ in the complementary set. From Lemma \ref{L:Delta-epsilon} and \eqref{E:holder-d} we see that $I$ is majorized by a multiple of
 \begin{equation*}
 \int\limits_{\d(w,\zp)<\epsilon^{1/2}}\!\!\!\!\!\!\!\!\!
 \d(w, z)^{-2n+\alpha}\,d\l(w),
 \end{equation*}
 which is $O(\epsilon^{\alpha/2})$ by \eqref{E:int-est-a}. 
 
 To estimate $II$ we note that
 $|\Delta(w, z_\epsilon)^{-n} - \Delta  (w, \zp)^{-n}|= O(\epsilon\,|\Delta(w, z)|^{-n-1})$. This estimates holds because of the inequality
  \begin{equation}\label{E:cplx-est-a}
 |u^{-n}-v^{-n}|\lesssim |u-v|\, |v|^{-n-1}\, ,
 \end{equation}
which is valid for any pair of complex numbers $u$ and $v$, with $|u|\gtrsim |v|>0$, taken together with Lemma \ref{L:Delta-epsilon} and the fact that $|\Delta(w, z_\epsilon) - \Delta  (w, \zp)|= O(\epsilon)$. Thus, by \eqref{E:int-est-b}, $II$ is  majorized by a multiple of
\begin{equation*}
\epsilon\!\!\!\!\!\!\!\!\!\!\int\limits_{\d(w,\zp)>\epsilon^{1/2}}\!\!\!\!\!\!\!\!
\d(w, z)^{-2n - 2+\alpha}\,d\l(w)\ =\ O(\epsilon^{\alpha/2})\quad \mbox{as long as}\ \alpha<2.
\end{equation*}
Therefore \eqref{E:unif-conv-a} is proved. The uniform convergence that follows from it ensures that $\CI (f)$ extends to a continuous function on $\overline{\D}$, thus proving conclusion {\em (c)} of the proposition, and furthermore we have
\begin{equation}\label{E:CI-CT}
\CI(f)\big|_{\bndry\D} =\CT (f).
\end{equation}
whenever $f$ satisfies the H\"older-like condition \eqref{E:holder-d}, and
we shall see below that the function $\CT (f)$ also satisfies condition \eqref{E:holder-d}.
\section{Main Theorem}\label{S:main-thm} 
Our goal here
 is to state the main result and to exhibit the fundamental feature of the Cauchy-Leray kernel that will lead to the cancellation conditions in Section \ref{SS:cancellation}.
\subsection{Statement.}\label{SS:statem}
 Let $\D$ be a bounded $C^{1,1}$ domain that is strongly $\C$-linearly convex. We consider the Cauchy-Leray transform $\CT$ that is given in terms of the Cauchy-Leray integral $\CI$ by \eqref{E:CLI-def} and \eqref{E:CI-CT}, or equivalently by \eqref{E:limit-a}.
\begin{theorem}\label{T:main}
The transform $f\mapsto \CT(f)$, initially defined for functions
  $f$ that satisfy the H\"older-like condition \eqref{E:holder-d} for some $\alpha$, extends to a bounded linear operator on $L^p(\bndry\D)$, for $1<p<\infty$.
\end{theorem}
Here $L^p(\bndry\D)$ is taken with respect to the induced Lebesgue measure on $\bndry\D$, or equivalently with respect to the Leray-Levi measure. 
\subsection{The Cauchy-Leray kernel as a derivative}\label{CL-ker-der}
The proof of the theorem relies on two aspects of the operator $\CT$, viewed as a ``singular integral''. First, a certain requisite regularity of the kernel of $\CT$ (that is, the function
 $\Delta (w, \zp)^{-n}$ for $w, \zp\in\bndry\D$) and a proof of this is straight-forward. However the second aspect, key ``cancellation conditions'', is more subtle. Here the turning point is a basic identity that
in effect expresses the Cauchy-Leray kernel as an appropriate derivative. Remarkably such an identity can hold only for $n>1$, because a one-dimensional analogue would necessarily involve a logarithmic term, invalidating its use below.
The identity is expressed in terms of two operators, $\Eop$ (an essential part) and $ \Rop$ 
(a remainder). The operator $\Eop$ acts on one-forms $\omega$ on $\bndry\D$ that have continuous coefficients; it maps these to continuous functions on $\overline\D$, and is defined by
\begin{equation}\label{E:Sop-def}
\Eop (\omega) (z)= c_n\!\!\!\!\int\limits_{w\in\bndry\D}\!\!\!\!
\Delta(w, z)^{-n+1}\,\omega\wedge j^*(\deebar\dee\rho)^{n-1},\quad z\in\overline{\D}
\end{equation}
where $c_n = 1/[(n-1)(2\pi i)^n]$. Observe that the singularity of the kernel of 
\eqref{E:Sop-def} is better by a factor of $\Delta (w, z)$ than that of the Cauchy-Leray integral, and so in particular the integral is absolutely convergent, by \eqref{E:int-est-a}.

The ``remainder'' operator $\Rop$ maps continuous functions on $\bndry\D$ to continuous functions on $\overline\D$ and is defined by
\begin{equation}\label{E:Eop-def}
\Rop (f) (z)= \frac{1}{(2\pi i)^n}\!\!\!\int\limits_{w\in\bndry\D}\!\!\!\!
\Delta(w, z)^{-n}\,f(w)\, j^*(\prec\alpha (w), w-z\!\succ)\wedge j^*(\deebar\dee\rho)^{n-1},\quad z\in\overline{\D}
\end{equation}
where now $\prec\alpha (w), w-z\!\succ = \sum\limits_{j=1}^n(w_j-z_j)\,\alpha_j(w)$ designates a 1-form with each $\alpha_j $ an appropriate 1-form on $\bndry\D$ with bounded measurable coefficients (to be specified below). Note that since
 $|\Delta (w, z)|^{-n}|w-z|\lesssim \d(w, z)^{-2n +1}$ by \eqref{E:useful}, the integral 
 \eqref{E:Eop-def} is also absolutely convergent. The basic identity in question is as follows.
 \begin{proposition}\label{P:basic-identity}
 If $f$ is a $C^1$ function on $\bndry\D$ then
 \begin{equation*}\label{E:weak-basic-identity}
 \CT(f) (\zp)= \Eop (df) (\zp)+ \Rop (f)(\zp), \quad \zp\in\bndry\D.
 \end{equation*}
 \end{proposition}
 Note that by \eqref{E:CI-CT}, the proposition is in fact a consequence of the identity
 \begin{equation}\label{E:basic-identity}
 \CI(f) (z)= \Eop (df) (z)+ \Rop (f)(z), \quad z\in\overline\D.
 \end{equation}
 To prove \eqref{E:basic-identity}, we start with the formal identity that reveals the extent to which the Cauchy-Leray kernel is a derivative. For $z\in\D$ and $n>1$, we have
 \begin{equation*}
 (1-n)^{-1}\, d_w\big(\Delta (w, z)^{-n+1}\big) =
\Delta(w, z)^{-n}\left\{\dee\rho (w) + \sum\limits_{j=1}^n(w_j-z_j)\, 
d_w\!\left(\frac{\dee\rho}{\dee w_j}(w)\right)\right\}
 \end{equation*}
  where $d_w$ is the exterior derivative on $\Cn$. (Here we regard $w$ as a variable in $\Cn\setminus \{z\}$.) This follows once we recall that
 \begin{equation*}
 \Delta (w, z) =\sum\limits_{j=1}^n\frac{\dee\rho (w)}{\dee w_j} (w_j-z_j),\quad \mbox{and}\quad 
 \dee\rho (w) =\sum\limits_{j=1}^n\frac{\dee\rho (w)}{\dee w_j} dw_j\, .
 \end{equation*}
 Applying the pullback via the inclusion $j: \bndry\D \hookrightarrow \Cn$ we obtain
 \begin{equation}\label{E:ker-deriv}
 \Delta(w, z)^{-n}j^*\dee\rho (w) = (1-n) j^* d_w\big(\Delta(w, z)^{-n+1}\big) -
 \Delta(w, z)^{-n}\sum\limits_{j=1}^n(w_j-z_j)\,\alpha_j(w),
  \end{equation}
where now $w$ is in $\bndry\D$ and $d_w$ denotes the exterior derivative for $\bndry\D$ viewed as a manifold in its own right,  and
 \begin{equation*}
 \alpha_j (w)\, =\, j^* d_w\!\left(\frac{\dee\rho(w)}{\dee w_j}\right).
 \end{equation*}
 Observe that for $z$ near $w$ the second term on the right-hand side of \eqref{E:ker-deriv} is negligible in size compared to the left-hand side. It is in this sense that the Cauchy-Leray kernel is a derivative.
 
 To utilize \eqref{E:ker-deriv} we use the approximations $\{\rho_k\}$ of the $C^{1,1}$ function 
 $\rho$ given us by Proposition \ref{P:2} and by Corollary \ref{C:2}; we also use Proposition
 \ref{P:approx-LF} applied to the Lipschitz functions $\dee\rho/\dee w_j $, $1\leq j\leq n$. 
 With $z$ fixed in $\D$ we know that $\Delta_k(w, z) =\langle\dee\rho_k (w), w-z\rangle\neq 0$
 for all $w\in\bndry\D$ and $k$ sufficiently large. Hence the analogue of \eqref{E:ker-deriv} holds 
 with $\rho$ replaced by $\rho_k$, and $\Delta (w, z)$ replaced by $\Delta_k(w, z)$, 
 and with 
 $\alpha_{j, k}$ defined as $\alpha_j$ but with $\rho_k$ in place of $\rho$. We now take the resulting version of \eqref{E:ker-deriv} and wedge it with
  $f(w)\,j^*(\deebar\dee\rho_k)^{n-1}$, where 
 $f$ is a $C^1$ function on $\bndry\D$. Note that
 \begin{equation}\label{E:deriv}
 (1-n)\,[\,j^*d_w\big(\Delta_k(w, z)^{-n+1}\big)]\wedge j^*(\deebar\dee\rho_k(w))^{n-1} = d_w\,\omega_k
 \end{equation}
 with $\omega_k= j^*[\Delta_k(w, z)^{-n+1}(\deebar\dee\rho_k(w))^{n-1}]$,
  because $\deebar\dee = d\dee$, $d_w\circ d_w=0$ and $j^*d_w=d_w j^*$.
Next we integrate over $\bndry\D$ and carry out the integration by parts (Stokes' theorem on the manifold $M=\bndry\D$, which has empty boundary) for the first term in the right-hand side of the
(``wedged'') $k$-analogue of \eqref{E:ker-deriv} expressed in terms of  \eqref{E:deriv}.
 The result is
 \begin{equation}\label{E:post-Stokes}
 \int\limits_{w\in\bndry\D}\!\!\!\!
 f(w)\,\Delta_k(w, z)^{-n}j^*(\dee\rho_k\wedge(\deebar\dee\rho_k)^{n-1})(w) =
 \end{equation}
 \begin{equation*}
 \frac{1}{(n-1)}\!\!\!\!\int\limits_{w\in\bndry\D}\!\!\!\!\!\Delta_k(w, z)^{-n+1}(d_wf)\wedge j^*(\deebar\dee\rho_k)^{n-1} +\!\!\!
  \int\limits_{w\in\bndry\D}\!\!\!\!\!
  \Delta_k(w, z)^{-n}f(w)\prec\alpha_k(w), w-z\!\succ\wedge j^*(\deebar\dee\rho_k)^{n-1}\!
 \end{equation*}
 with $\prec\alpha_k(w), w-z\!\succ = \sum\limits_{j=1}^n(w_j-z_j)\alpha_{j, k}(w)$.
 
 Now a passage to the limit as $k\to\infty$ gives us the analogue of \eqref{E:post-Stokes}, with
 $\rho_k$ replaced by $\rho$, $\Delta_k(w, z)$ replaced by $\Delta(w, z)$ and $\alpha_k(w)$
 replaced by $\alpha (w)$. This shows that the identity
 \begin{equation*}
 \CI(f)(z)= \Eop (df)(z) + \Rop (f)(z)
 \end{equation*}
 holds for $z\in\D$. A second passage to the limit, using that $F(z) = \CI(f)(z)$ is continuous on $\overline{\D}$ then establishes this for $z\in\bndry\D$, as well, and the proposition is proved.
 \subsection{Cancellation conditions}\label{SS:cancellation} These are expressed in terms of the
action of $\CT$ on certain ``test'' functions. We fix some $\alpha>0$ and say that a function 
$f$ defined on $\bndry\D$ is a {\em normalized bump function associated to a ball $\B_r = \B_r(\wp)$} (with $\wp\in\bndry\D$), if
\begin{equation}\label{E:bf-a}
f\ \ \mbox{is supported in}\ \B_r\, ,
\end{equation}
 \begin{equation}\label{E:bf-b}
 |f(w)|\leq 1\ \ \mbox{for all}\ w\in\bndry\D
\end{equation}
and
\begin{equation}\label{E:bf-c}
|f(w_1)- f(w_2)|\leq \left(\frac{\d(w_1, w_2)}{r}\right)^\alpha\ \ \mbox{for all}\ w_1, w_2\in\bndry\D.
\end{equation}
\begin{proposition}\label{P:nbf}
If $f$ is a normalized bump function, then
\begin{equation}\label{E:nbf-a}
\sup\limits_{z\in\bndry\D}|\CT (f) (z)|\lesssim 1\, , \qquad \mbox{and}
\end{equation}
\begin{equation}\label{E:nbf-b}
\|\CT (f)\|_{L^2(\bndry\D, \,d\l)}\lesssim r^{n}\, .
\end{equation}
\end{proposition}
The second inequality asserts that insofar as $\CT$ is tested on these bump functions, it is bounded on $L^2$. (Note that $\l (\B_r)\approx r^{2n}$, by Proposition \ref{P:surface-balls}.)

The proof of the proposition will require that we first prove a conclusion of the same sort but under more restrictive conditions.
\begin{lemma}
\label{L:nbf}
Suppose $f_0$ is a $C^1$ function on $\bndry\D$ that is supported in $\B_r(\wp)$ and that satisfies the following conditions: $|f_0|\leq 1$ and $|\nabla f_0|\leq 1/r^2$. 
Then $|\CT(f_0)(\wp)|\lesssim 1$.
\end{lemma}
The proof of this lemma is a direct consequence of Proposition \ref{P:basic-identity}: if we set
$\zp=\wp$, we see by \eqref{E:Sop-def} that 
\begin{equation*}
|\Eop (df_0) (\wp)|\leq \frac{c_n}{r^2}\!\!\!\!\!\!\!\int\limits_{w\in\B_r(\wp)}\!\!\!\!\!\!\!
\d(w, \wp)^{-2n+2}\, d\l (w)
\end{equation*}
since $|df_0(w)|\lesssim |\nabla f_0(w)|\lesssim 1/r^2$. As a result $|\Eop (df_0)(\wp)|\leq c$, in view of \eqref{E:int-est-a} with $\epsilon =2$. Next, by \eqref{E:Eop-def},
\begin{equation*}
|\Rop (f_0) (\wp)|\leq c_n\!\!\!\int\limits_{w\in\bndry\D}\!\!\!\!
|\Delta (w, \wp)|^{-n}|w-\wp|\,d\l(w)\leq c'\, .
\end{equation*}
This follows if we use the facts that $|\Delta (w, z)|=\d(w, z)^2$ and $|w-z|\leq c\,\d(w,z)$, together with \eqref{E:int-est-a} with $r=1$ and $\epsilon=1$. Therefore Lemma \ref{L:nbf} is established.
 
 We turn to the proof of Proposition \ref{P:nbf}. We prove first the key conclusion:
 \begin{equation}\label{E:key-a}
 |\CT (f)(\wp)|\lesssim 1
 \end{equation}
 that is, the result when $z$ is the center of the ball $\B_r$. To do this, we choose a coordinate system 
 as in Section \ref{SS:strong C-lin-convex}, now centered at $\wp$, and so the coordinates at $\wp$ are $(0,0,\ldots, 0)$. 
  Next we choose a $C^1$ function $\chi_0$ on $\mathbb C$, so that $\chi_0(u)=1$, if $|u|\leq 1/2$ and $\chi_0(u) =0$ if $|u|\geq 1$, 
 and in these coordinates we set $\chi (w) = \chi_0(w_n/cr^2)$.
 Here $c$ is a small constant that will be chosen below. Then $\chi (w)$ is supported where $|w_n|\leq cr^2$, and because $|\Delta (w, \wp)|\approx |w_n|$, see  \eqref{E:Delta-repr} and Proposition \ref{P:quasi-metric}, and
 $|\Delta (w, \wp)|=\d(w, \wp)^2$, we see that $\chi (w)$ is supported in $\B_r(\wp)$, once we take $c$ sufficiently small. Also, if we calculate $|\nabla\chi (w)|$ in these coordinates 
  we see clearly that $|\nabla\chi(w)|\lesssim 1/r^2$.
 
 We now define $f_0$ by $f_0(w) =f(\wp)\chi(w)$. Then it follows that up to a constant multiple, the function $f_0$ satisfies the hypotheses of the lemma just proved. However $\CT(f)=
 \CT(f_0) + \CT(f-f_0)$, and $|\CT(f_0)(\wp)|\lesssim 1$ by the lemma. Next, since $f(\wp) - f_0(\wp)=0$, we have by \eqref{E:limit-a} that
 \begin{equation}\label{E:aux-repr-a}
 \CT(f-f_0)(\wp)= \int\limits_{w\in\bndry\D}\!\!\! \frac{f(w)-f_0(w)}{\Delta(w, \wp)^n}\, d\l(w)\, .
 \end{equation}
 Now $f(w)-f_0(w) = f(w) - f(\wp) + f(\wp)[1-\chi(w)]$. However by assumption,
 $|f(w) - f(\wp)|\leq (\d(w, \wp)/r)^\alpha$, while
  $|1-\chi(w)|\lesssim |w_n|/r^2\approx \d(w, \wp)^2/r^2$, and the latter is majorized by
  $(\d(w, \wp)/r)^\alpha$, 
     for $w\in\B_r(\wp)$. Also, both $f$ and $f_0$ are supported in $\B_r(\wp)$, so that the integral \eqref{E:aux-repr-a} is  majorized by a multiple of
  \begin{equation*}
  r^{-\alpha}\!\!\!\!\int\limits_{w\in\B_r(\wp)}\!\!\!\!\!
  \frac{d\l(w)}{\d(w,\wp)^{2n-\alpha}},
  \end{equation*}
  which is bounded in $r$, by \eqref{E:int-est-a}. Therefore \eqref{E:key-a} is proved.
  
  We consider next the case when $z$ is such that $\d(\wp, z)\leq c_1r$, with $c_1$ a large constant to be chosen below. In this case, for such $z$, we may think of our given $f$ as a 
  multiple of a normalized bump function associated to a ball $\B_{c_2r}(z)$, centered at $z$, for an appropriate choice of $c_2$. In fact, if $\d (\wp, w)\leq c_1r$ and $\d(z, \wp)\leq r$, the triangle inequality $\d(w, z)\leq c(\d(\wp, z) + \d(w, \wp))$, shows that $\d(w, z)\leq c_2r$, with $c_2=c(c_1+1)$. So $f$ is supported in $\B_{c_2r}(z)$. Also a constant multiple of 
  $f$ clearly satisfies the requisite inequalities in the hypotheses of Proposition \ref{P:nbf}. Thus by the special case just proved, we have
  \begin{equation*}
 | \CT (f) (z) |\lesssim 1,\quad \mbox{if}\ \ \d(\wp, z)\leq c_1r.
  \end{equation*}
  Finally we consider the case when $z$ is such that $\d(\wp, z)>c_1r$. 
  By the triangle inequality again, 
  $\d(\wp, z)\leq c(\d(w, z) +\d(w, \wp))$, we obtain
  \begin{equation}\label{E:aux-ineq-b}
  \d(w, z)\gtrsim \d(\wp, z),\quad \mbox{if}\ \ \d(\wp, z)\geq c_1\d(\wp, w),\  \mbox{with}\ c_1=2c
  \end{equation}
   Since $\d(\wp, w)\leq r$ when $w$ is in $\B_r(\wp)$ (the support of $f$), then if $\d(z, \wp)\geq c_1r$, the inequality 
  \eqref{E:aux-ineq-b} shows us that $\d(w, z)\gtrsim \d(\wp, w)$ for that $z$. So in this case, since $f(z)=0$, the formula \eqref{E:limit-a-out} grants that 
  \begin{equation*}
  \CT(f)(z) =\int\limits_{w\in\bndry\D}\!\!\!\!
  \frac{f(w)}{\Delta(w, z)^n}\, d\l (w)\, ,
  \end{equation*}
  and since $|\Delta(w, z)|=\d(w, z)^2$, this shows that $|\CT (f) (z)|$ is  majorized by a multiple of 
  \begin{equation*}
  \frac{\lambda(\B_r(\wp))}{\d(\wp, w)^{2n}}\, .
    \end{equation*}
  As a consequence
  \begin{equation}\label{E:aux-ineq-c}
  |\CT(f)(z)|\lesssim r^{2n}\d(\wp, w)^{-2n},\quad \mbox{if}\ \ \d(\wp, w)\geq c_1r\, .
  \end{equation}
  The right-hand side of \eqref{E:aux-ineq-c} is in particular bounded by 1, and this completes the proof of \eqref{E:nbf-a}. The proof of \eqref{E:nbf-b} follows immediately: we write
  \begin{equation*}
  \int\limits_{z\in\bndry\D}\!\!\!\!
  |\CT (f)(z)|^2\,d\l (z) = I+II
  \end{equation*}
  where $I$ is the integral over $\B_{c_1r}(\wp)$ and $II$ the integral over the complement. By 
    \eqref{E:nbf-a}, we have
   $|I|\lesssim \l (\B_{c_1r}(\wp))\approx r^{2n}$. For $II$ one has that this integral is dominated by a multiple of 
  \begin{equation*}
  r^{4n}\!\!\!\!\!\!\!\!\!\!\int\limits_{z\in\B_{c_1r}(\wp)^c}\!\!\!\!\!\!\!\!\!
  \d(\wp, z)^{-4n}\, d\l (z)
  \end{equation*}
  which in turn is  majorized by a multiple of $r^{2n}$ by \eqref{E:int-est-b} (for $\epsilon = 2n$). Combining the estimates for $I$ and $II$ proves \eqref{E:nbf-b} and hence the proposition.\\
  
  \noindent{\bf Remark.}
 The argument above can also be used to show that
  $$\sup\limits_{z\in\bndry\D}|\CT (f_0)(z)|\lesssim 1$$ for $f_0$ as in Lemma \ref{L:nbf}.\\

   At this point we record a further cancellation condition, one already used before, namely
  \begin{equation}\label{E:cancellation}
  \CT (1) = 1.
  \end{equation}
  (See, for example
     \eqref{E:limit-a}.)
  \subsection{Kernel estimates}\label{SS:ker-ests}
  Let us take $K(w, z)$ to be the function defined for $w, z\in\bndry\D$, with $w\neq z$, by
  \begin{equation*}
  K(w, z) = \frac{1}{\Delta(w, z)^n}\, .
  \end{equation*}
  This function is the ``kernel'' of the operator $\CT$, in the sense that
  \begin{equation*}
  \CT (f) (z) = \int\limits_{w\in\bndry\D}\!\!\!\!
  K(w, z)f(w)\,d\l (w)\, ,
  \end{equation*}
  whenever $z$ lies outside of the support of $f$, and $f$ satisfies \eqref{E:holder-d}. This is evident from \eqref{E:limit-a}. The size and regularity estimates that are relevant for us are:
  \begin{equation}\label{E:ker-est-a}
  |K(w, z)|\lesssim \d (w, z)^{-2n}.
  \end{equation}
  \begin{equation}\label{E:ker-est-b}
  |K(w_1, z) -K(w_2, z)|\lesssim \d(w_1, z)^{-2n}\,\frac{\d(w_1, w_2)}{\d(w_1, z)}\
   \quad \mbox{if}\ \
  \d(w_1, z)\geq c_1\d(w_1, w_2)
  \end{equation}
  \begin{equation}\label{E:ker-est-c}
  |K(w, z_1) -K(w, z_2)|\lesssim \d(w, z_1)^{-2n}\,\frac{\d(z_1, z_2)}{\d(w, z_1)}\
    \quad \mbox{if}\ \
  \d(w, z_1)\geq c_1\d(z_1, z_2)\, .
  \end{equation}
  The assertion \eqref{E:ker-est-a} is obvious since $|\Delta(w, z)|= \d(w, z)^2$. To establish the other facts we first verify the simple inequality:
  \begin{equation}\label{E:ker-est-d}
  |\Delta (w_1, z) -\Delta (w_2, z)|\lesssim \d(w_1, w_2)^2 + \d(w_1, w_2)\,\d(w_1, z).
  \end{equation}
  Indeed, 
  $\langle\dee\rho(w_1), w_1-z\rangle - \langle\dee\rho(w_2), w_2-z\rangle =
  \langle\dee\rho (w_2), w_1-w_2\rangle +
  O(|w_1-w_2|\,|w_1-z|)$,  since $\dee\rho$ is Lipschitz.
  
  Inequality \eqref{E:ker-est-d} then follows, since
  $|\langle\dee\rho (w_2), w_1-w_2\rangle| = \d(w_2, w_1)^2$, while $|w_1-w_2|\lesssim \d(w_1, w_2)$ and $|w_1-z|\lesssim \d(w_1, z)$, by \eqref{E:useful}. Next, since $\d(w_1, z)\geq c_1\d(w_1, w_2)$, the triangle inequality (Proposition \ref{P:quasi-metric})
   implies that 
    $|\Delta (w_2, z)|\gtrsim |\Delta (w_1, z)|$. If we now use \eqref{E:cplx-est-a} with
  $u= \Delta (w_2, z)$ and $v=\Delta (w_1, z)$ we obtain that 
  \begin{equation*}
  |K(w_1, z) -K(w_2, z)| 
   \lesssim
  |\Delta (w_1, z) -\Delta (w_2, z)|\, |\Delta(w_1, z)|^{-n-1}\, .
  \end{equation*}
  Therefore \eqref{E:ker-est-d} and \eqref{E:dist-def}
    yield that the above is  majorized by 
 \begin{equation*}
 \d(w_1, z)^{-2n}
 \left[\frac{\d(w_1, w_2)}{\d(w_1, z)}+ \left(\frac{\d(w_1, w_2}{\d(w_1, z)}\right)^2\right]
 \lesssim
  \d(w_1, z)^{-2n}\frac{\d(w_1, w_2)}{\d(w_1, z)}
 \end{equation*}
 since $\d(w_1, w_2)/\d(w_1, z)\lesssim 1$. Hence \eqref{E:ker-est-b} is proved. The proof of
 \eqref{E:ker-est-c} is parallel to that of \eqref{E:ker-est-b}. 
  
 \section{Proof of the main theorem}\label{S:proof-MT}
 Here we conclude the proof of the main theorem, and in particular we deal with the (appropriate)
  adjoint of the Cauchy-Leray transform.
 \subsection{The adjoint of $\CT$}\label{SS:adj-CT} Having set down some basic properties of the Cauchy-Leray transform $\CT$, we turn to the actual proof of its $L^p$-boundedness. It is essential here that we establish for its ``adjoint'' $\CT^*$, 
  certain properties parallel to those proved for $\CT$. Here and in the manipulations below it is crucial that we are using the Leray-Levi measure $d\l$, and not the induced Lebesgue measure $d\sigma$; however this distinction will not matter in the final statements of the $L^p(\bndry\D)$-boundedness of $\CT$. 
We will define this purported adjoint as a limit, and so we consider $\CT$ itself as a limit, namely 
 $\displaystyle{\CT=\lim\limits_{\epsilon\to 0} \CT_\epsilon}$ where 
 $\CT_\epsilon (f) (\zp) = \CT f(z_\epsilon)$ for $z_\epsilon = \zp+ \epsilon\nu_{\zp}$ as in Section \ref{SS:repr-hol-fncts},
 with $\zp\in\bndry\D$ and $\nu_{\zp}$ the inner unit normal at $\zp$. That is, we take 
 $\CT_\epsilon$ to be given by the operator 
 \begin{equation*}\label{E:def-CT-eps}
 \CT_\epsilon (f)(\zp) =\int\limits_{w\in\bndry\D}\!\!\!\!  \Delta(w, z_\epsilon)^{-n}f(w)\,d\l (w)\, .
 \end{equation*}
 Since $|\Delta (w, z_\epsilon)|\gtrsim \epsilon$ for every $\epsilon>0$ sufficiently small (see Lemma \ref{L:Delta-epsilon}), the operator $\CT_\epsilon$ is bounded on every $L^p(\bndry \D, d\l)$ together with its genuine adjoint $\CT_\epsilon^*$, which 
  satisfies
 \begin{equation}\label{E:CT-eps-adj}
 \left(\CT_\epsilon (f), g\right) = \left(f, \CT_\epsilon^* (g) \right)\quad \mbox{ for all}\ \ f, g\in L^2(\bndry\D, d\l)
 \end{equation}
 where
 \begin{equation*}
 \left(f, g\right) = \int\limits_{w\in\bndry\D}\!\!\! 
 f(w)\,\overline{g} (w)\, d\l (w)\, ,
 \end{equation*}
 and is given by 
 \begin{equation*}\label{E:def-adj-CT-eps}
 \CT_\epsilon^* (f) (\zp) = \int\limits_{w\in\bndry\D}\!\!\!  
 \overline{\Delta}  (\zp, w_\epsilon)^{-n} f(w)\, d\l (w)\, 
 \end{equation*}
 with $w_\epsilon = w+ \epsilon\nu_{w}$. 
 Our first goal is to show that the limit, $\displaystyle{\lim_{\epsilon\to 0}\CT_\epsilon}^*$, exists. For this we need the analogues for $\CT_\epsilon^*$ of the operators $\Eop$ and $\Rop$ of Section \ref{CL-ker-der}, which we will label, respectively,  $\Eopa$ and $\Ropa$. The operator $\Eopa$ is defined as 
 \begin{equation*}
 \Eopa (\omega) (\zp) = c_n \int\limits_{w\in\bndry\D}\!\!\!
 \overline{\Delta}  (\zp, w)^{-n+1}\omega\wedge j^*(\deebar\dee\rho)^{n-1}(w),\quad \zp\in\bndry\D,
 \end{equation*}
 for any continuous $1$-form $\omega$ on $\bndry\D$, with $c_n = 1/[(n-1) (2\pi i)^n]$.
  Also 
 \begin{equation}\label{E:def-Eopa}
 \Ropa (f) (\zp) = \frac{1}{(2\pi i)^n}\!\!\!\int\limits_{w\in\bndry\D}\!\!\!\!
\overline{\Delta}  (\zp, w)^{-n}\,f(w)\, j^*\big(\prec d\,\!\overline{w}, \beta  (\zp, w)\!\succ\big)\wedge j^*(\deebar\dee\rho)^{n-1}(w)
 \end{equation}
 for $\zp\in\bndry{\D}$, where now $\prec d \overline{w},\beta  (\zp, w)\rangle$ designates the 1-form
 \begin{equation}\label{E:def-beta}
 \prec d\,\!\overline{w},\beta  (\zp, w)\!\succ \ = \ 
 \sum\limits_{j=1}^n\left(\frac{\dee\rho (w)}{\dee\overline{w}_j} - 
 \frac{\dee\rho (\zp)}{\dee\overline{z}_j}
 \right) d\overline{w}_j\, .
 \end{equation}

Note that the integrals defining $\Eopa$ and $\Ropa$ converge absolutely for the same reasons that those defining $\Eop$ and $\Rop$ do.
\begin{proposition}\label{P:basic-identity-a} Suppose $f\in C^1(\bndry\D)$. Then $\lim_{\epsilon\to 0}\CT_\epsilon^* (f)(\zp)$ converges uniformly for $\zp\in\bndry\D$. If we denote $\CT^*(f)(\zp)$ this limit, then 
\begin{equation}\label{E:weak-basic-identity-adj}
 \CT^*(f) (\zp)= \Eopa (df) (\zp)+ \Ropa (f)(\zp), \quad \zp\in\bndry\D.
\end{equation}
\end{proposition}
The proof follows the same lines as that of Proposition \ref{P:basic-identity} but the details are a
 bit different. One begins by observing that the following decomposition
 \begin{equation*}
 \CT^*_\epsilon (f) (\zp) = \Eopa_\epsilon (df)(\zp) + \Ropa_\epsilon (f) (\zp)\, 
 \end{equation*}
 can be proved in the same way as the decomposition for $\CT$ given in Proposition \ref{P:basic-identity},
 where
 \begin{equation}\label{E:Sopa-eps-def}
 \Eopa_\epsilon (\omega)(\zp) = 
 c_n\!\!\!\int\limits_{w\in\bndry\D}\!\!\!\overline{\Delta}  (\zp, w_\epsilon)^{-n+1}\omega\wedge j^*(\deebar\dee\rho)^{n-1}(w)
 \end{equation}
 and $\Ropa_\epsilon (f) (\zp) = \Ropa_\epsilon^{\,(1)} (f) (\zp) + \epsilon\,
  \Ropa_\epsilon^{\,(2)} (f) (\zp)$, with
 \begin{equation}\label{E:Eopa-eps-1-def}
 \Ropa_\epsilon^{\,(1)} (f) (\zp) = 
 \frac{1}{(2\pi i)^n}\!\!\!\int\limits_{w\in\bndry\D}\!\!\!
\overline{\Delta}  (\zp, w_\epsilon)^{-n}\,f(w)\, 
j^*\big(\,\prec d\overline{w},\beta  (\zp, w)\!\succ\,\big)
\wedge j^*(\deebar\dee\rho)^{n-1}(w)
 \end{equation}
 with $\prec d\overline{w},\beta  (\zp, w)\!\succ$ the $1$-form that was given
  in \eqref{E:def-beta}, and
 \begin{equation*}\label{E:Eopa-eps-2-def}
 \Ropa_\epsilon^{\,(2)} (f) (\zp) =
 \frac{1}{(2\pi i)^n}\!\!\!\int\limits_{w\in\bndry\D}\!\!\!\!
\overline{\Delta}  (\zp, w_\epsilon)^{-n}\,f(w)\, 
j^*\big(\,\prec d\nu_w, \gamma(\zp)\!\succ\,\big)\wedge
 j^*(\deebar\dee\rho)^{n-1}(w)
 \end{equation*}
 with
 \begin{equation*}
 \prec d\nu_w, \gamma(\zp)\!\succ =
 \sum\limits_{j=1}^n\frac{\dee\rho}{\dee\overline{z}_j}(\zp)\,d\nu_j(w)
 \end{equation*}
where $d\nu_w = (d\nu_1 (w), \ldots, d\nu_n (w))$, so that $j^*(d\nu_j(w))$, $j=1,\dots, n$, are $1$-forms on $\bndry \D$ with bounded measurable coefficients. Now let us see what happens when $\epsilon\to 0$. First, if the form $\omega$ is
 bounded, then \eqref{E:Sopa-eps-def} shows that $\Eopa(\omega)(\zp) -\Eopa_\epsilon(\omega)$ is dominated by a multiple of
 \begin{equation*}
 \int\limits_{w\in\bndry\D}\!\!\!
 \big|\overline{\Delta}(\zp, w)^{-n+1} - \overline{\Delta} (\zp, w_\epsilon)^{-n+1}\big|\, d\l (w)\, .
 \end{equation*}
 We break this integral into the two regions where $\d(w, z)<\epsilon^{1/2}$ and the complementary region. In the first region we use Lemma \ref{L:Delta-epsilon}. In the second region we note that
  $\overline{\Delta}(\zp, w) - \overline{\Delta} (\zp, w_\epsilon) = O(\epsilon)$ and use the inequality \eqref{E:cplx-est-a} (with $n$ replaced by $n-1$, and with $u= \overline{\Delta} (\zp, w_\epsilon)$; $v= \overline{\Delta}(\zp, w)$)
  after invoking again Lemma \ref{L:Delta-epsilon}. The outcome is that the above integral is bounded by a multiple of
  \begin{equation*}
   \int\limits_{\d(w,\,\zp)\leq \epsilon^{1/2}}\!\!\!\!\!\!\!\!\!\!
  \d  (\zp, w)^{-2n+2}\,d\l (w) +
  \epsilon\!\!\!\!\!\!\!\!\!
   \int\limits_{\d (w,\,\zp)> \epsilon^{1/2}}\!\!\!\!\!\!\!\!\!\!
  \d  (\zp, w)^{-2n}\,d\l (w)\, ,
  \end{equation*}
  and this is $O(\epsilon + \epsilon\log(1/\epsilon))$, as $\epsilon\to 0$, by Corollary \ref{C:int-ests} in Section \ref{SS:integral-estimates}. Thus in particular $\Eopa_\epsilon (df) \to \Eopa (df)$
  uniformly on $\bndry\D$ as $\epsilon\to 0$.
  
  Next, by \eqref{E:Eopa-eps-1-def}, if $f$ is bounded, then 
  $\Ropa (f) (\zp) - \Ropa_\epsilon^{\,(1)}f(\zp)$ is dominated by a multiple of
   \begin{equation*}
 \int\limits_{w\in\bndry\D}\!\!\!
 \big|\overline{\Delta}(\zp, w)^{-n} - \overline{\Delta} (\zp, w_\epsilon)^{-n}\big|\, \d (\zp, w)\,d\l (w)\, 
 \end{equation*}
 since $|\deebar\rho (\zp) -\deebar\rho (w)|\lesssim \d (\zp, w)$. Thus by a very similar argument to that given just above, this is majorized by 
 \begin{equation*}
 \int\limits_{\d(w,\,\zp)\leq \epsilon^{1/2}}\!\!\!\!\!\!\!\!\!\!
  \d  (\zp, w)^{-2n+1}\,d\l (w) +
  \epsilon\!\!\!\!\!\!\!\!\!
   \int\limits_{\d (w,\,\zp)> \epsilon^{1/2}}\!\!\!\!\!\!\!\!\!\!
  \d  (\zp, w)^{-2n-1}\,d\l (w)
  \end{equation*}
which is $O(\epsilon^{1/2} + \epsilon^{1/2}) = O(\epsilon^{1/2})$. As a result, 
$\Ropa_\epsilon^{\,(1)}f\to \Ropa f$ uniformly on $\bndry\D$ as $\epsilon\to 0$. Finally, for bounded $f$, 
$\Ropa_\epsilon^{\,(2)}f$  is dominated by a multiple of 
$\int\! |\Delta  (\zp, w_\epsilon)|^{-n}\, d\l (w)$, and by Lemma \ref{L:Delta-epsilon} this is majorized
 by
  \begin{equation*}
 \epsilon^{-n}\!\!\!\!\!\!\!\!\!\!\!
 \int\limits_{\d(w,\,\zp)\leq \epsilon^{1/2}}\!\!\!\!\!\!\!\!\!\!
  d\l (w) +
   \int\limits_{\d (w,\,\zp)> \epsilon^{1/2}}\!\!\!\!\!\!\!\!\!\!
  \d  (\zp, w)^{-2n}\,d\l (w)\, ,
  \end{equation*}
  which is $O(\log(1/\epsilon))$ as $\epsilon\to 0$. The proposition is therefore proved.\\
  
  Note: it follows from the proposition just proved and from \eqref{E:CT-eps-adj} that 
  \begin{equation}\label{E:CT-duality}
  (\CT (f), g) = (f, \CT^* (g))\quad \mbox{for any}\ \ f, g\in C^1(\bndry\D)
  \end{equation}
  with $f\mapsto\CT^*(f)$ as in \eqref{E:weak-basic-identity-adj}.
  \begin{corollary}\label{C:basic-identity-a}
  We have that $\mathfrak{h} := \CT^*(1)$ is a continuous function that satisfies the H\"older-like condition
  \eqref{E:holder-d} for all $0<\alpha<1$.
  \end{corollary}
  Indeed, by \eqref {E:weak-basic-identity-adj} we have $\CT^*(1) = \Ropa (1)$. Now, looking back
  at \eqref{E:def-Eopa} we see that $|\Ropa (1) (\zp_1) - \Ropa (1) (\zp_2)|$ is majorized by a 
  multiple of 
  \begin{equation}\label{E:majoriz-1}
   \int\limits_{w\in\bndry\D}\!\!\!
 \big|\overline{\Delta}(\zp_1, w)^{-n}\Nop(\zp_1, w) - 
 \overline{\Delta}(\zp_2, w_\epsilon)^{-n}\Nop (\zp_2, w)\big|\, d\l (w)\, 
  \end{equation}
  where $\Nop  (\zp, w) = \deebar\rho (\zp) - \deebar\rho (w)$. We then break the integration in 
  \eqref{E:majoriz-1} into two parts: where $\d  (w, z_1)\leq c_1\d (z_1, z_2)$ and where
  $\d (w, z_1) > c_1 \d (z_1, z_2)$. Over the first region the integral is majorized by 
  \begin{equation}\label{E:majoriz-2}
  \sum\limits_{j=1, 2}\ \,\int\limits_{\d (w,\,\zp_j)\leq c\,\d(\zp_1, \zp_2)}\!\!\!\!\!\!\!\!\!\!\!\!\!\!\!
 \d  (\zp, w)^{-2n}\,\d (\zp_j, w)\, d\l (w)\, ,
   \end{equation}
  if $c$ is a sufficiently large constant, since 
  \begin{equation}\label{E:estimate-1}
  |\Nop (\zp_j, w)|
   \lesssim |z_j-w|\lesssim \d(\zp_j, w)\, .
  \end{equation}
   Thus the contribution of \eqref{E:majoriz-2} is $O(\d (\zp_1, \zp_2))$.
   Over the second region we write the integrand in \eqref{E:majoriz-1} as 
  \begin{equation*}
  \Nop (w, \zp_1)\,\big[\,\overline{\Delta}(w, \zp_1)^{-n} - \overline{\Delta}(w, \zp_2)^{-n}\big] +
  \big[\Nop (w, \zp_1) -\Nop (w, \zp_2)\big]\,\overline{\Delta}(\zp_1, w)^{-n}\, .
  \end{equation*}
  We estimate this quantity using \eqref{E:cplx-est-a} and the estimate \eqref{E:estimate-1} (applied to $\zp_1$) for the first term; also, 
 we have
$  |\Nop(\zp_1, w) - \Nop (\zp_2, w)|\lesssim |\zp_1-\zp_2|\lesssim \d(\zp_1, \zp_2)$ for the second term. Hence altogether our integral is dominated by a multiple of
 \begin{equation*}
 \int\limits_{\d(w,\,\zp_1)\leq c\,\d(\zp_1, \zp_2)}\!\!\!\!\!\!\!\!\!\!\!\!\!\!\!\!
  \d (w, \zp_1)^{-2n+1}\,d\l (w)\ +\
  \d(\zp_1, \zp_2)\!\!\!\!\!\!\!\!\!\!\!\!\!\!\!\!\!
   \int\limits_{\d (w,\,\zp_1)> c_1\d(\zp_1, \zp_2)}\!\!\!\!\!\!\!\!\!\!\!\!\!\!\!\!
  \d (w, \zp_1)^{-2n}\,d\l (w)
  \end{equation*}
  and this is $O\big(\d(\zp_1, \zp_2)\,\log(1/\d(\zp_1, \zp_2))\big) = O(\d(\zp_1, \zp_2)^\alpha)$, for any $0<\alpha<1$. The corollary is thus proved.
\subsection{Further properties of $\CT^*$}\label{SS:further-CT-adj} As a consequence of the corollary just proved we have the following analogue of \eqref{E:limit-a}. Whenever $f$ is a function on $\bndry\D$ that satisfies the Ho\"lder-like condition \eqref{E:holder-d} for some $\alpha>0$, then $\CT^*_\epsilon (f)$ converges uniformly on $\bndry\D$ to a limit, which we also denote $\CT^*(f)$, and this is given by
\begin{equation}\label{E:limit-adj-a}
\CT^*(f)(\zp) =
\int\limits_{w\in\bndry\D}\!\!\!
 \frac{f(w)-f(\zp)}{\overline{\Delta} (\zp, w)^n}\, d\l (w) + \mathfrak{h}(\zp)\cdot f(\zp), \quad \zp\in\bndry\D\, .
\end{equation}
The proof is word for word almost the same as that of 
\eqref{E:limit-a} once we write
\begin{equation*}
\CT_\epsilon^*(f)(\zp) =
\int\limits_{w\in\bndry\D}\!\!\!
 \frac{f(w)-f(\zp)}{\overline{\Delta} (\zp, w_\epsilon)^n}\, d\l (w) + \big(\CT^*_\epsilon(1)(\zp)\big)\cdot f(\zp), \quad \zp\in\bndry\D
\end{equation*}
and use the fact that $\CT^*_\epsilon(1)(\zp)\to \mathfrak{h}(\zp)$ uniformly on $\zp\in\bndry\D$ by Proposition \ref{P:basic-identity-a}. 

A final parallel of $\CT^*$ with $\CT$ 
is a version of Proposition \ref{P:nbf} and its corollary that gives the action of $\CT^*$ on normalized bump functions.
\begin{proposition}\label{P:nbf-adj} Suppose $f$ is a normalized bump function, satisfying 
properties\ \ \ \ \eqref{E:bf-a} -- \eqref{E:bf-c}. Then
\begin{equation*}\label{E:nbf-adj-a}
\sup\limits_{\zp\in\bndry\D}|\CT^* (f) (\zp)|\lesssim 1\, ,\quad and\qquad
\|\CT^*(f)\|_{L^2(\bndry\D, d\l)}\lesssim r^n\, .
\end{equation*}
\end{proposition}
Once we have the identities \eqref{E:weak-basic-identity-adj} and \eqref{E:limit-adj-a} the proof of this proposition is nearly identical with that of Proposition \ref{P:nbf} and its corollary, and will therefore be omitted.
\subsection{Application of the $T(1)$ theorem}\label{SS:T(1)} We invoke the extended $T(1)$ theorem in the context of spaces of homogeneous type of Coifman, as in \cite{DJS}, \cite{C-1} and \cite{C-2}. Following \cite[Chapter VI]{C-1}, we identify 
\begin{equation*}
X\ \ \mbox{with}\ \ \bndry\D\ \ \ \mbox{and}\ \ \ x, y\in X\quad \mbox{with}\ \  w, z\in \bndry\D,\ \mbox{respectively}.
\end{equation*}
Correspondingly,
\begin{equation*}
\rho(x, y):= \d(w, z);\ \ \ \mu(B(x, r)) :=\lambda(\B_r(z))\approx r^{2n};\ \ \ K(x, y) :=\Delta(w, z)^{-2n}
 \end{equation*}
 satisfy the conditions in \cite[Definition 6]{C-1} with $\epsilon:=\alpha$, for any 
 $0<\alpha\leq 1$. 
 
 Also, 
  $T:=\CT$ satisfies the conditions of Definitions 6 and 7 in \cite{C-1}.\\
 
 At this point it is convenient to modify $T$ by considering, instead:
 \begin{equation*}
 T_0(f) (\zp)\ =\ \CT(f)(\zp) -\frac{1}{\lambda_0}\!\!\!
 \int\limits_{w\in\bndry\D}\!\!\!\overline{\mathfrak{h}}(w)\, f(w)\, d\l (w)
 \end{equation*}
 where again $\mathfrak{h} = \CT^* (1)$, and $\lambda_0 =\l (\bndry\D)$. Note that by Corollary
  \ref{C:basic-identity-a}, the boundedness of $\CT$ in $L^p(\bndry\D)$ is equivalent to the boundedness 
  in $L^p(\bndry\D)$ of $T_0$. \\
  
  Also note that $(1, \mathfrak{h}/\lambda_0)=1$ (since $(\CT (1), 1) = (1, \CT^*(1))= (1, \mathfrak{h})$, and $\CT(1) =1$, by \eqref{E:cancellation} and \eqref{E:CT-duality}). Thus, the operator $T_0$ with kernel $K_0 := \Delta(w, \zp)^{-n} - \overline{\mathfrak{h}}(w)/\lambda_0$ satisfies the same properties as $T$ above, with the additional feature that $T_0(1) = T_0^* (1) =0$. Hence by \cite[Theorem 13]{C-1}, we conclude that $T_0$, and therefore $T$, is bounded:
  $L^p(\bndry\D)\to L^p(\bndry\D)$.
   
 \subsection{Further results} One has the following analog of the classical theorem of Privalov. 
 \begin{proposition}\label{P:further-holder}
For any $0<\alpha<1$, the transform $f\mapsto \CT(f)$ preserves the class of H\"older-like functions satisfying condition \eqref{E:holder-d}.
\end{proposition}
To prove the proposition we need to show that $|\CT(f)(\zp_1)-\CT(f)(\zp_2)|\lesssim \d(\zp_1, \zp_2)^\alpha$ for any $\zp_1, \ \zp_2\in\bndry\D$, whenever $f$ satisfies
this same condition. Fix $\zp_1\in\bndry\D$ and consider the boundary ball $\B_r(\zp_1) =\{ w\in\bndry D\ :\  \d (\zp_1, w)<r\}$ with radius $r= C\,\d (\zp_1, \zp_2)$ where $C$ is a sufficiently large constant, and let $\chi_{\zp_1}(w)$ be the special cutoff function, supported in this ball, that was constructed in the proof of Proposition \ref{P:nbf} (with the center now at $\wp=\zp_1$). One decomposes each of 
$\CT(f)(\zp_1)$ and $\CT(f)(\zp_2)$ as follows:
\begin{equation*}
\CT(f)(\zp_j) = I_j + II_j + f(\zp_j),\quad j=1, 2
\end{equation*}
where
\begin{equation*}
I_j=\int\limits_{w\in\bndry\D}\!\!\!\!\!
K(w,\zp_j)\chi_{\zp_1}(w)(f(w)-f(\zp_j))\,d\l(w)
\end{equation*}
and
\begin{equation*}
II_j= \int\limits_{w\in\bndry\D}\!\!\!\!\!
K(w,\zp_j)(1- \chi_{\zp_1}(w))(f(w)-f(\zp_j))\,d\l(w)
\end{equation*}
with $K(w, \zp_j) = \Delta (w, \zp_j)^{-n}$. 
The first observation is then that each of $|I_1|$ and $|I_2|$ is majorized by a constant multiple of $\d(\zp_1, \zp_2)^\alpha$ (this is because the integrands are majorized by $\d(w, \zp_j)^{-2n+\alpha}$, and then one uses \eqref{E:int-est-a} with $\epsilon =\alpha$.) Next one shows that $|II_1 - II_2|$ is also majorized by 
a constant multiple of $\d(\zp_1, \zp_2)^\alpha$. To see this, one further decomposes the term $II_2$ as follows
\begin{equation*}
II_2 = \widetilde{II}_2 +
 +
(f(\zp_1)-f(\zp_2))\!\!\!\!\!\int\limits_{w\in\bndry\D}\!\!\!\!\!
K(w,\zp_2)(1- \chi_{\zp_1}(w))\,d\l(w)
\end{equation*}
with
\begin{equation*}
\widetilde{II_2} =  \int\limits_{w\in\bndry\D}\!\!\!\!\!
K(w,\zp_2)(1- \chi_{\zp_1}(w))(f(w)-f(\zp_1))\,d\l(w)\, .
\end{equation*}
Then one observes that the difference $|II_1 - \widetilde{II}_2|$ is easily seen to be majorized by
$\d(\zp_1, \zp_2)^\alpha$ via the kernel estimate for $|K(w, \zp_1) - K(w, \zp_2)|$, see \eqref{E:ker-est-b}, and the integral estimate \eqref{E:int-est-b} with $\epsilon =1$.
Finally, the integral in the remaining term, 
$\int
K(w,\zp_2)(1- \chi_{\zp_1}(w))\,d\l(w)$, is uniformly bounded, because $\CT (1) =1$ and $\CT(\chi_{\zp_1})$ is uniformly bounded by the remark following the proof of Lemma \ref{L:nbf}.
 The proof of Proposition \ref{P:further-holder} is concluded.
\section{Appendix I: Proof of some approximation lemmas}\label{S:Appr-lemm}
The proof of Lemma \ref{L:2} is more-or-less standard. We recall only what's involved in the proofs of $(c)$ and the first part of $(b)$.
We have 
$$\displaystyle \nabla^2 g_k (x) = \int_{\mathbb{R}^{N-1}}\limits\!\!\!  \nabla^2\!g\, (x - y) \,\eta_{\nicefrac{1}{k}} (y)\, dy,$$
 where $\nabla^2 g$ is the matrix of $L^\infty$ functions that arise as the second derivatives of $g$, taken in the sense of distributions. As a result, $\nabla^2 g_k(x) \to \nabla^2 g(x)$ for a.e. $x \in \mathbb{R}^{N-1}$, by the $(N-1)$-dimensional Lebesgue differentiation theorem. 
To see the first assertion in $(b)$, write 
$$\displaystyle g(x) - g_k (x) = \int_{\mathbb{R}^{N-1}}\limits\!\!\! (g(x) - g(x-y))\, \eta_{\nicefrac{1}{k}} (y)\, dy.$$
\noindent However, using
Lemma \ref{L:1}, and
taking into account that 
$\int\limits_{\mathbb R^{N-1}}\!\!\! y_j\, \eta_{\nicefrac{1}{k}}(y)\,dy = 0$ (because $\eta$ is even), we get the estimate
\begin{equation*}
 | g(x) - g_k (x) | \leq c\!\! \int_{\mathbb{R}^{N-1}}\limits\!\!\! | y |^2\, | \eta_{\nicefrac{1}{k}} (y) |\, dy \ .                                                                                                                                                                 
\end{equation*}
 The conclusion
 $ \| g - g_k \| = O(\nicefrac{1}{k^2} )$
    then follows because
\begin{equation*}
   \int_{\mathbb{R}^{N-1}}\limits\!\!\! | y |^2 | \eta_{\nicefrac{1}{k}} (y) |\, dy\  = 
   \ k^{N-1}\!\!\!\!\! \int\limits_{\mathbb{R}^{N-1}}\!\!\! | y |^2\, | \eta (ky) |\, dy \ = \ 
   k^{-2} \!\!\!\!\!\int\limits_{\mathbb{R}^{N-1}}\!\!\!| y |^2\, |\eta (y) | \,dy \ .
\end{equation*}

We turn to the proof of Lemma \ref{L:3}. It is based on the following four observations about the functions $(\eta_{\nicefrac{1}{k}} \ast \eta_{| y |})(x)$.
First,  as functions of $x \in \mathbb{R}^{N-1}$, these are supported in the ball $| x | < \nicefrac{1}{k} + | y |$. This is clear because $\eta_{\nicefrac{1}{k}}$ and $\eta_{| y |}$ are supported in the balls $| x | < \nicefrac{1}{k}$ and $| x | < | y |$, respectively.
 Second, 
for each $k, \ ( \eta_{\nicefrac{1}{k}} \ast \eta_{| y |} ) (x)$ is a $C^\infty$-smooth function of $(x, y) \in \mathbb{R}^N$. For this, and some further observations, we avail ourselves of the Fourier inversion formula to write 
\begin{equation}\label{E:(1.6)}
( \eta_{\nicefrac{1}{k}} \ast \eta_{| y |} ) (x) = \int_{\mathbb{R}^{N-1}}\limits\!\!\! e^{2 \pi i x \cdot \xi} \,\widehat{\eta} (\nicefrac{\xi}{k})\, \widehat{\eta} (y \xi)\, d \xi \ ,  
\end{equation} 
where $\widehat{\eta}$ is the Fourier transform of $\eta$. (Here again we have used the fact that $\eta$, and hence $\widehat{\eta}$ are even.) From \eqref{E:(1.6)} and the rapid decay and regularity of $\widehat{\eta}$, the $C^\infty$ character of $\left( \eta_{\nicefrac{1}{k}} \ast \eta_{| y |}\right) (x)$ is evident.
Third, 
one has the estimates
\begin{equation}\label{E:(1.7)}
   \begin{cases}
   \sup_x\limits D^2 \left( \eta_{\nicefrac{1}{k}} \ast \eta_{| y |}\right) (x) &\lesssim  
 \min\left( k^{N+1} , | y |^{-N-1} \right)\\
   \sup_x\limits D \left( \eta_{\nicefrac{1}{k}} \ast \eta_{| y |}\right) (x) &\lesssim  \min \left( k^{N} , | y |^{-N} \right)
     \end{cases}
\end{equation} 
Indeed, by examining \eqref{E:(1.6)} we see that any second derivative of $\eta_{\nicefrac{1}{k}} \ast \eta_{| y |}$ brings down a quadratic factor in $\xi$ and possibly replaces the $\widehat{\eta}$ by other rapidly decreasing functions. As a result, 
\begin{equation}\label{E:(1.8)}
   \left| D^2\!\left( \eta_{\nicefrac{1}{k}} \ast \eta_{| y |} \right) \right| \lesssim \int_{\mathbb{R}^{N-1}}\limits | \xi |^2 \left| \Phi \left( \nicefrac{\xi}{k} \right) \right| d \xi
\end{equation} 
where $\Phi$ is rapidly decreasing. The integral is a constant multiple of $k^{N+1}$, as a change of scale shows. By the same observation we can replace the right side of \eqref{E:(1.8)} by $
\int\!
 | \xi |^2 \left| \Phi \left( \xi y \right) \right| d \xi$, which is a multiple of $| y |^{-N-1}$. This establishes the first inequalities in \eqref{E:(1.7)}; the others are proved in the same way. 
Finally, 
we have the following identities.
\begin{equation}\label{E:aux-3}
\displaystyle \int\limits_{\mathbb{R}^{N-1}}\!\!\!( \eta_{\nicefrac{1}{k}} \ast \eta_{| y |})(x)\, dx = 1\,;\qquad 
\displaystyle \int\limits_{\mathbb R^{N-1}}\!\!\!\!  D\!  \left( \eta_{\nicefrac{1}{k}} \ast \eta_{| y |} \right) dx = 0\,;
\end{equation}
\begin{equation}\label{E:aux-6a}
\int\limits_{\mathbb R^{N-1}}\!\!\!\!
 \dee^2_{yy}\!\left( \eta_{\nicefrac{1}{k}} \ast \eta_{| y |} \right)
 dx \ =\ 0\ =
\int\limits_{\mathbb R^{N-1}}\!\!\!\! x_j\,
 \dee_{y}\! \left( \eta_{\nicefrac{1}{k}} \ast \eta_{| y |} \right)
dx \, \ j=1,\ldots, N-1\, ;
\end{equation}
\begin{equation}\label{E:aux-5}
 \int\limits_{\mathbb R^{N-1}}\!\!\!\! x_j\, 
  \dee_{x_i}(\eta_{\nicefrac{1}{k}} \ast \eta_{| y |})
= \left\{ \begin{array}{rcl}
\ 0&\mbox{if}\quad i \neq j\\
-1&\mbox{if}\quad i=j\, .
\end{array}\right.
\end{equation}
The first identity in \eqref{E:aux-3} holds because $\int\!\eta_{\nicefrac{1}{k}}(x) dx =
 \int \!\eta_{| y |}(x) dx = 1$; the first of the two identities in \eqref{E:aux-6a}
  follows upon differentiating the first identity in \eqref{E:aux-3} with respect to $y$. The other identities
   are direct consequences of 
\eqref{E:aux-3} and integration by parts.

Armed with these observations, we now prove Lemma \ref{L:3}. 
We focus on the statements for $Dg_k$, since the statements for $g_k$ are proved in the same way and in fact are easier. First,
\begin{equation*}
   D g_k (x, y) = \int\limits_{\mathbb{R}^{N-1}}\!\!\! \left(g(x-u) - g(x)\right)\! D\!\left( \eta_{\nicefrac{1}{k}} \ast \eta_{| y |} \right)\!(u)\, du
\end{equation*}
because of \eqref {E:aux-3}. Next the integral is estimated by 
\begin{equation*}
  O \bigg(\int\limits_{| u | \leq | y |+1/k}\!\!\!\!\!\!\!\! | u |\, du \bigg)\!
    \cdot \min \left\{ k^N ,\, | y |^{-N} \right\}.
\end{equation*}
This is because $g(x-u) - g(x) = O( | u | )$; the function $\left( \eta_{\nicefrac{1}{k}} \ast \eta_{| y | } \right) (u)$ is supported in $| u | \leq | y |+\nicefrac{1}{k}$; and the second estimate in 
\eqref{E:(1.7)}. However 
$$\displaystyle \int\limits_{| u | \leq | y |+1/k}\!\!\!\!\!\!\!\!  | u |\,du = O \left( | y |+1/k \right)^N,\quad \mbox{while}\quad \left(| y |  +1/k\right)^N\!\! \cdot \min \left\{ k^N , | y |^{-N} \right\} = O(1),$$ so the assertion $\| D g_k ( \cdot , y ) \| = O (1)$ is established. 
The conclusion $\| D^2 g_k (\cdot , y) \| = O(\min \{k, | y |^{-1}\})$ is proved the same way. Finally, coming to conclusion {\em (d)}, we 
use Rademacher's theorem to assert that 
$g(x-u) - g(x) = (\nabla g(x), u)  + o(|u|)$ for a.e. $x\in\mathbb R^{N-1}$, because $g$ is Lipschitz. We insert this fact in the integral below
\begin{equation*}
   \dee_{x_j}g_k  (x,y) = \int \!\!g(x-u)\,
 \dee_{u_j} \!\! \left( \eta_{\nicefrac{1}{k}} \ast \eta_{| y |}  \right)\! (u,y)\, du ,
\end{equation*}
and using 
\eqref{E:aux-3} and \eqref{E:aux-5},
 we obtain
\begin{equation*}
   \dee_{x_j} g_k (x,u) - \dee_{x_j} g(x) = o
    \bigg( k^{N-1}\! \!\!\!\int\limits_{| u | < 2/k}\!\!\!\!\! \!| u |\, du \bigg) =\ o (1 ) \ , \quad \text{ as } \ k \rightarrow \infty,
\end{equation*}
as long as $| y | \leq \nicefrac{1}{k}$, thus showing  that 
$\dee_{x_j} g_k (x,y) \rightarrow \dee_{x_j} g (x)$ 
for almost every $x$.
The conclusion $\dee_y\, g_k (x, y)\rightarrow 0$ a.e. $x$ is proved similarly, using the second identity in \eqref{E:aux-6a}.

 \section{Appendix II: The implicit function theorem}\label {S:IFT}
Here we give a quantitative version of the implicit function theorem for $C^{1,1}$
functions that was referred to earlier on. Suppose $F(x, y)$ is a $C^{1,1}$ function for $x\in\mathbb R^{N-1}$, $y\in\mathbb R$. Assume $\|F\|_{C^{1,1}}\leq M$. Suppose also $F(0, 0) =0$,
$\dee_yF (0,0) =1$.
\begin{proposition}
There are $a, b>0$, $a=a(M),\ b=b(M)$ so that in the box $\{|x|\leq a,\ |y|\leq b\}$, the equation 
$F(x, y)=0$ has a unique solution for each $x$, with $|x|\leq a$. This defines $y=\varphi (x)$.
Then the function $\varphi (x)$ is of class $C^{1,1}$ (for $|x|\leq a$), and 
$\|\varphi\|_{C^{1,1}}\leq \mathfrak{m}$, with $\mathfrak{m}= \mathfrak{m}(M)$.
\end{proposition}
\begin{proof}
First consider the box $B_1 = \{|x|\leq a_1, \ |y|\leq a_1\}$ with $a_1$ chosen momentarily. Since $\dee_yF$ is Lipschitz, in this box, $|\dee_yF(x, y) - \dee_yF (0,0)|\leq M(|x|+|y|)\leq 2M a_1$. Hence if we take 
$2Ma_1 = 1/2$ then, using $\dee_yF(0, 0) =1$ we obtain
\begin{equation}\label{E: F_y-est-1}
\dee_yF(x, y)\geq 1/2,\quad \mbox{throughout}\ \ B_1.
\end{equation}
Next choose $a\leq a_1$, to be determined below. So if $|x|\leq a$, then by the mean-value theorem
$|F(x, 0)| = |F(x, 0) - F(0, 0)|\leq Ma$. However 
\begin{equation*}
F(x, y_+) = F(x, 0) + \int\limits_0^{y_+} \!\! \dee_yF( x, u)\,du\geq -Ma +y_+/2\,\geq\, 0\quad \mbox{if}\ \
y_+ =\, 2Ma.
\end{equation*}
Similarly  $F(x, y_-)\leq 0$, if $y_-=-2Ma$. Hence
 $F(x, y) =0$ has a unique solution $y=\varphi (x)$ that lies in the interval 
$[-2Ma, 2Ma]$ that is, $|\varphi (x)|\leq 2Ma$ (note that this solution is unique since $F(x, y)$ is strictly increasing in $y$, see \eqref{E: F_y-est-1}). We now consider our final box 
$B=\{|x|\leq a,\ |y|\leq b\}$ by setting 
\begin{equation*}
a= \frac{a_1}{2M} = \frac{1}{8M^2}\qquad \mbox{and}\qquad b= 2Ma = \frac{1}{2M}=a_1\, .
\end{equation*}

Then clearly $B\subset B_1$ and $(x, \varphi (x))\in B$ whenever $|x|\leq a$ (note that $M\geq 1$ automatically, since $\dee_yF(0, 0)=1$). Next, if $|x|\leq a$ and $|x+h|\leq a$, then
\begin{equation}\label{E:F-decomp}
F(x, \varphi (x+h)) - F (x, \varphi (x)) = 
F(x, \varphi (x+h)) - F(x+h, \varphi (x+h))
\end{equation}
since $F(x, \varphi (x)) = 0$ and $F(x+h, \varphi (x+h))=0$. However by the mean-value 
theorem, the lefthand side of \eqref{E:F-decomp} equals $\delta\cdot \dee_yF(x, y')$, where $\delta =\varphi (x+h) - \varphi (x)$ and $y'$ is a point on the segment joining $\varphi (x+h)$ to $\varphi (x)$. Thus by \eqref{E: F_y-est-1}, the absolute value of the lefthand side of \eqref{E:F-decomp} exceeds $|\delta|/2$. However since $F$ is Lipschitz with bound $M$, the righthand side of 
\eqref{E:F-decomp} is majorized by $M |h|$. As a result, $\varphi$ is Lipschitz and 
$|\varphi (x+h) -\varphi (x)| = |\delta| \leq 2M|h|$. Next we re-examine \eqref{E:F-decomp} using the fact that $F$ is of class $C^{1,1}$. We then see that the lefthand side is $\delta \dee_yF(x, \varphi (x)) + O(|\delta|^2)$, and the righthand side is $-(\nabla_{\! x}F(x, \varphi (x), h) + O(|h|^2)) + O(|\delta|\cdot |h|)$. But since $|\delta|=O(|h|)$, as we have seen, this yields (using again \eqref{E: F_y-est-1})
\begin{equation*}
\varphi (x+h)-\varphi (x) = -\frac{1}{\dee_yF(x, \varphi (x))}\, (\nabla_{\! x}F(x, \varphi (x)), h) + O(|h|^2)\, .
\end{equation*}
If we let $h\to 0$ we see that clearly $\varphi$ has all its first derivatives at each $x$, and
\begin{equation*}
\nabla\varphi (x) = -\frac{1}{\dee_yF(x, \varphi (x))}\, \nabla_{\! x}F(x, \varphi (x))
\end{equation*}
As a result, we see that $\nabla\varphi$ is a Lipschitz function since it arises as the composition of Lipschitz functions. In particular, $\nabla_{\!x}\,F (x, y)$ has Lipschitz norm
 majorized by $M$ and $\varphi$ has Lipschitz norm $2M$. Hence $\nabla_{\!x} F(x, \varphi (x))$ has
Lipschitz norm  majorized by $2M^2$. For similar reasons, $1/\dee_yF(x, \varphi (x))$ has Lipschitz norm less than or equal to $4\cdot 2M^2= 8M^2$. Altogether the Lipschitz norm of $\nabla\varphi$ is bounded above by $CM^4$, and this shows that $\|\varphi\|_{C^{1,1}}\leq C'M^4 =:\mathfrak{m}$.
\end{proof}
\noindent{\bf Remark 1.}
If our assumption were $\dee_yF(0, 0) = m\neq 0$ (instead of $m=1$), then by replacing $F$ by $F/m$ we would get a similar conclusion, with $M$ replaced by $M/|m|$. 
\noindent Also note that if we further assumed that $\nabla_{\!x }F(0, 0)=0$ then we would have $\nabla\varphi (0) =0$. \\

\noindent{\bf Remark 2.}
We  note that in the above we can write $F(x, y) = A(x, y) (y-\varphi (x))$, where $A$ is a Lipschitz function and $|A(x, y)|\geq |m|/2$.\\

 To see that this is true, write 
\begin{equation*}
F(x, y) =\int\limits_0^1\!\frac{d}{ds} \left(F(x, s(y-\varphi (x)) +\varphi (x)\right)ds.
\end{equation*} 
Then we take
\begin{equation*}
A(x, y) =\int\limits_0^1\!
\dee_y\!
\left(F(x, s(y-\varphi (x)) +\varphi (x)\right)ds,
\end{equation*} 
so $A$ is Lipschitz, and $A\geq 1/2$ (if $m=1$) by \eqref{E: F_y-est-1}. From this it is easy to prove that if $\rho$ and $ \rho'$ are a pair of $C^{1,1}$ defining functions for the same bounded domain $\D$, then $\rho' = a\rho$, where $a$ is a Lipschitz function with $a(w)>0$ for each $w\in\bndry\D$.

\end{document}